\documentclass[11pt]{article}

\usepackage{lastpage,extramarks} 
\usepackage{amsmath,amsthm,thmtools,mathtools,amssymb,mathrsfs,graphicx,wrapfig,enumitem} 
\usepackage{setspace}
\usepackage[english]{babel}
\usepackage[x11names]{xcolor}
\usepackage{tikz}
\usetikzlibrary{angles,quotes}
\usepackage[colorlinks=true,linkcolor=.,citecolor=.]{hyperref}
\usepackage{imakeidx,hyperref}
\usepackage{pdfpages}
\usepackage[shortcuts]{extdash}

\theoremstyle{plain}

\theoremstyle{definition}
\newtheorem{defn}{Definition}[section]
\newtheorem{lem}[defn]{Lemma}
\newtheorem{prop}[defn]{Proposition}
\newtheorem{cor}[defn]{Corollary}

\theoremstyle{remark}
\newtheorem{ex}[defn]{Example}
\newtheorem{rem}[defn]{Remark}

\newcommand{\com}[1]{}

\newcommand{\IN}{\mathbb{N}}
\newcommand{\IZ}{\mathbb{Z}}

\newcommand{\IR}{\mathbb{R}}

\newcommand{\IS}{\mathbb{S}}
\newcommand{\eps}{\varepsilon}

\DeclareMathOperator\diam{diam}
\DeclareMathOperator\dist{dist}

\DeclareMathOperator\MinCut{MinCut}
\DeclareMathOperator\Cut{Cut}
\DeclareMathOperator\UltConj{UltConj}
\DeclareMathOperator\conj{conj}

\DeclareMathOperator\MinRad{minrad}
\DeclareMathOperator\inj{inj}

\DeclareMathOperator\id{id}

\title{The Injectivity Radius of Souls of Alexandrov Spaces}
\author{Elena M\"ader-Baumdicker\footnote{maeder-baumdicker@mathematik.tu-darmstadt.de,\newline ORCID: 0000-0001-8125-8799,\newline Technical University Darmstadt, Department of Mathematics, Schlossgartenstr. 7, 64289 Darmstadt, Germany}, Jona Seidel\footnote{seidel@mathematik.tu-darmstadt.de,\newline ORCID: 0009-0000-1282-8209,\newline Technical University Darmstadt, Department of Mathematics, Schlossgartenstr. 7, 64289 Darmstadt, Germany }}
\date{\today}

\makeindex

\begin{document}
\maketitle

\begin{abstract}
	A sharp lower bound for the injectivity radius in noncompact nonnegatively curved Riemannian manifolds involving their soul goes back to \v{S}a\-ra\-fut\-di\-nov. We generalize this bound to the setting of Alexandrov spaces. Our main theorem reads as follows. If the injectivity radius of an Alexandrov space of nonnegative curvature does not coincide with the one of its souls, then it is at least $ \pi K^{-1/2} $, where $ K $ is an upper curvature bound. We introduce the soul of Alexandrov spaces in some detail and compare two notions of injectivity radii.	\footnote{Mathematics Subject Classification \href{https://mathscinet.ams.org/mathscinet/search/mscdoc.html?code=53C20,53C23,53C45}{53C20; 53C23, 53C45},\newline Keywords: Alexandrov Geometry, Soul, Injectivity Radius, Conjugate Points
	}
\end{abstract}

\newpage
\section*{Acknowledgments}
We are thankful to Alexander Lytchak for extensive discussions which were essential to our understanding and improved the exposition of this paper at various places. We are also thankful to Karsten Große-Brauckmann for his helpful comments and commitment to our fruitful discussions. Special thanks should also go to Fernando Galaz-García and Diego Corro who together with the first author suggested the formulation of Theorem~\ref{thm_injSoul} and shared helpful comments. The first author is partially supported by the Deutsche Forschungsgemeinschaft DFG (MA 7559/1-2). This work is part of the second author's PhD thesis.
\newpage
\tableofcontents

\section{Introduction}
\com{I omitted the following historical remarks.
	History of AS:
	\begin{itemize}
		\item cite{Alex} and cite{Bus} and cite{Rin}.
		\item Don't overlook remarks by Plaut \href{https://mathscinet.ams.org/mathscinet/search/publdoc.html?arg3=&co4=AND&co5=AND&co6=AND&co7=AND&dr=all&pg4=AUCN&pg5=TI&pg6=PC&pg7=AUCN&pg8=ET&review_format=html&s4=Shiohama&s5=alexandrov&s6=&s7=&s8=All&sort=Newest&vfpref=html&yearRangeFirst=&yearRangeSecond=&yrop=eq&r=3&mx-pid=1320267}{in his review} on cite{Shio}
	\end{itemize}
}
An Alexandrov space is a complete length space, where curvature bounds are introduced by triangle comparison. This defines a wide class of spaces including complete Riemannian manifolds and their limits \cite{BGP,BN,BBI}\com{Lower curvature bound cit[10.7.1]{BBI}. An upper curvature bound is not necessarily stable under limits. Additional assumptions are required cit[11.3]{BN}. For both-sided curvature bounds, see cite[15.1]{BN}}. The metric toolbox of Alexandrov geometry is different to the usual apparatus of smooth differentiable structures and builds upon basic descriptions of geometry like distance and angle. Many results from Riemannian geometry have been generalized to Alexandrov spaces. In fact, some questions appear to be more naturally described in the metric vocabulary and lead to new results even for Riemannian manifolds, e.g.\ see \cite{AB_GaussEqInjRadSubspaces,AKP_Found}. 

The concept of the soul was a critical contribution to the study of Riemannian manifolds of nonnegative curvature by Cheeger and Gromoll \cite{CG}. It is a compact submanifold without boundary that contains any (local) geodesic joining two of its points. The Soul Theorem \cite[Theorem~2.2]{CG} essentially reduces the study of complete non\-negatively curved Riemannian manifolds to the compact case. It states that these manifolds are diffeomorphic to the normal bundle over their soul.
The construction of the soul in \cite{CG} was generalized to Alexandrov spaces by Perelman \cite{P2}.
We call a soul which can be obtained by this construction a Cheeger-Gromoll soul. 
It was also shown in \cite{P2} that the full Soul Theorem does not generalize to Alexandrov spaces (see Example~\ref{ex_soulTheoremCounterexample}).
Still, the soul serves as a key tool to investigate the structure of nonnegatively curved Alexandrov spaces \cite{W,Li,HS_OrientAS,Spind,GPMaxRad}. For the convenience of the reader we include a section summarizing the construction of the soul in \cite{P2}. Moreover, we state a proof for the uniqueness of Cheeger-Gromoll souls up to isometry.

We show that, while the Soul Theorem does not generalize to Alexandrov spaces, a theorem by \v{S}arafutdinov \cite{Sh76} does.
It establishes a lower bound for the injectivity radius $ \inj(M) $ of a Riemannian manifold $ M $ in terms of an upper curvature bound and the injectivity radius of its soul. In Riemannian manifolds, each pair of points at a distance less than $ \inj(M) $ are uniquely joined by a geodesic, and each geodesic of length less than $ \inj(M) $ is minimizing. These properties provide two natural notions of an injectivity radius in Alexandrov spaces, cf.\ Section~\ref{sec_Inj}. Lower bounds on the injectivity radius are significant as they control the volume growth of balls \cite{CK}, retain curvature bounds under limits \cite[Prop.~11.3]{BN}, prevent collapsing or yield topological implications such as sphere theorems \cite[6.5]{Panorama}.

An alternative proof to \v{S}arafutdinov's theorem 
was given by Croke and Karcher \cite[Theorem I]{CK}. We adapt it to the setting of Alexandrov spaces and obtain the following as our main result.
\begin{restatable}{thm}{injSoul}\label{thm_injSoul}
	Let $ K> 0 $ and $ X $ be a finite-dimensional Alexandrov space of non\-negative curvature bounded above by $ K $ and $ S\subseteq X$ be a Cheeger-Gromoll soul. If $ \inj(X)<\frac{\pi}{\sqrt{K}} $, then 
	\begin{equation*}
		\inj(X) = \inj(S).
	\end{equation*}
\end{restatable}
The assumption $ \inj(X)< \frac{\pi}{\sqrt{K}} $ in Theorem \ref{thm_injSoul} is necessary. The injectivity radius of the paraboloid is finite. However, for its soul, a singleton, it is infinite.
\begin{rem}
	Note that, as a consequence of Theorem~\ref{thm_injSoul} we obtain the lower bound on the injectivity radius
	\begin{equation*}
		\inj(X)\geq \min \left\{\frac{\pi}{\sqrt{K}},\ \inj(S)\right\}.
	\end{equation*}
	In this form the theorem is stated in \cite{Sh76,CK}. However, the slightly stronger formulation in Theorem~\ref{thm_injSoul} also follows from their proofs.
\end{rem}
Notably, Theorem~\ref{thm_injSoul} can reduce the study of the injectivity radius to compact spaces without boundary, a case well studied in the Riemannian setting, see \cite[6.5.2]{Panorama}. In addition, further information on the soul can give a lower bound independent of quantities such as dimension, volume or diameter. The following corollary is an example of that.
It is known that in the case of positive curvature (see Definition~\ref{def_posCurv}), the soul is a single point unless it coincides with the entire space (see Proposition~\ref{prop_soulPoint}). A space consisting of a single point has infinite injectivity radius, and so we conclude:
\begin{restatable}{cor}{injSoulPos}\label{cor_injSoulPos}
	Let $ K>0 $ and $ X $ be a finite-dimensional Alexandrov space of positive curvature bounded above by $ K $. If $ X $ is non\-compact or has non\-empty boundary, then
	$$ \inj(X)\geq \frac{\pi}{\sqrt{K}}. $$
\end{restatable}
This contrasts to the case of compact spaces without boundary where curvature bounds are not sufficient to establish a lower bound for the injectivity radius \cite[6.5.2]{Panorama}.\com{Gather examples of compact positively curved Riemannian manifolds without boundary and arbitrary low injectivity radius (see hint at \cite[p.~277~bottom]{Panorama})}
\begin{rem}
	In Riemannian manifolds a single point with positive sectional curvatures suffices to guarantee a singleton soul. This is the Soul Conjecture which was proposed in \cite{CG} and famously proved in \cite{P_Soul}. In Alexandrov spaces this is still open, with some recent progress \cite{W,Li,RW}. In Alexandrov spaces where the Soul Conjecture continues to hold, it allows for the weaker curvature assumption in Corollary~\ref{cor_injSoulPos}.
	
\end{rem}
\begin{rem}
Although we work in the setting of Alexandrov spaces, the assumptions of Theorem~\ref{thm_injSoul} yield more regularity than one might expect. Two-sided curvature bounds imply that $ X $ and $ S $ are topological manifolds, possibly with boundary \cite{Pl92Boundary}. Moreover, $ S $ has empty Alexandrov boundary (see Definition~\ref{def_boundary}). Therefore, $ S $ has empty boundary in the manifold sense (Corollary~\ref{cor_boundaryManifold}) and hence satisfies local extendability of geodesics \cite[ch.\ II, 5.12]{BH}. This qualifies $ S $ for the theory of \cite{Ni_Overview,BN} which yields that $ S $ is in fact a Riemannian manifold, although not necessarily smooth \cite[14.1]{BN}.
Furthermore, it can be approximated by smooth Riemannian manifolds \cite[15.1]{BN}. However, in the present paper we do not use the Riemannian metric or any of these results. Instead, the metric tools of Alexandrov spaces suffice to conclude Theorem~\ref{thm_injSoul}.
\end{rem}

The role of the curvature bounds in Theorem~\ref{thm_injSoul} can be seen as follows. Nonnegative curvature provides the existence of the soul and thereby controls the length of closed geodesics. On the other hand, the upper curvature bound controls the distance to conjugate points. These two quantities determine the injectivity radius, see Lemma~\ref{lem_Klingenberg}.

The paper is organized as follows. After introducing basic facts  in Alexandrov geometry in Section~\ref{sec_Preliminaries}, we present in Section~\ref{sec_Soul} Perelman's existence proof of the soul and gather some properties of the soul construction. In Section~\ref{sec_Inj}, we introduce two definitions of an injectivity radius in Alexandrov spaces and treat conjugate points. These are the tools necessary to prove Theorem~\ref{thm_injSoul} in Section~\ref{sec_InjSoul}. Finally, Appendix~\ref{sec_App} collects elementary proofs and an addendum to \cite[Theorem~8.3]{ShSo}.

\section{Preliminaries}\label{sec_Preliminaries}
In this section we fix notation and recall basic properties of Alexandrov spaces. We mainly refer to \cite{BH,BBI,AKP_Found}. Other excellent introductions include \cite{BGP,Shio,Pl02}. 

In the following $ \kappa,K\in\IR $ denote real numbers and $ n $ a natural number. If not specified otherwise, $ d $ denotes the metric of a metric space $ X $. We denote an open ball of radius $ r\in\IR $ at $ p\in X $ by $ B_r(p) $. By \textit{reversal} of a curve $ c:[a,b]\to X $, we mean the curve $ -c:[a,b]\to X,\ t\mapsto c(a+b-t) $.
\begin{defn}
	Let $ X $ be a metric space.
	\begin{enumerate}[label=(\roman*)]
		\item A \textit{length space} is a metric space where the distance between two points is given by the infimum over lengths of all rectifiable curves joining those two points. 
		\item A curve $ c:[a,b]\to X $ is called a \textit{(local) geodesic} if it is a linear reparametrization of a (local) isometry. If it is a (local) isometry, then it is called a \textit{unit-speed (local) geodesic}. A curve $ c:[0,\infty)\to X $ is called a (unit-speed) \textit{ray} if each restriction to a compact interval is a nontrivial (unit-speed) geodesic.
		\item $ X $ is called \textit{geodesic} if for each pair of points there exists a geodesic joining them.
		\item A subset $ C\subseteq X $ is called \textit{convex} if any two points can be joint by a geodesic that lies entirely in $ C $.
	\end{enumerate}
\end{defn} 
Note that our local geodesics correspond to some authors notion of ``geodesics'' and our geodesics to ``minimizing geodesics''. While others require (local) geodesics to be parametrized by arc length (unit-speed), the above definition allows for parametrizations proportional to arc length.

Curvature bounds can be introduced in a length space $ X $ via comparison of triangles in $ X $ to triangles of the same side lengths in the simply connected $ 2 $-dimensional Riemannian manifolds $ M_K $ of constant curvature $ K $.
We denote its distance metric by $ d^K $ and set $ D_K $ to be its diameter, namely
\begin{equation*}
	D_K:=\begin{cases}
		\infty \quad &\text{ if } K\leq 0,\\
		\frac{\pi}{\sqrt{K}} &\text{ if } K>0 .
	\end{cases}
\end{equation*}
Triangles in $ M_K $ are well understood and described by the law of cosines \cite[Section~1.A]{AKP_Found}. For a choice of three pairwise distinct points $ p,q,r\in X $ satisfying $ d(p,q)+d(q,r)+d(r,p)<2D_K $ we denote by $ \tilde{p},\tilde{q},\tilde{r}\in M_K $ a choice of points such that each pair is the same distance apart as their counterparts in $ X $. By $ \tilde{\Delta}^K  $ we denote the triangle with vertices $ \tilde{p},\ \tilde{q}$ and $\tilde{r} $ and $ \tilde{\angle}^K_p(q,r) $ denotes the angle of $ \tilde{\Delta}^K $ in $ \tilde{p} $.
\begin{defn}\label{def_Curv}
	Let $ X $ be a geodesic space.
	$ X $ is a CAT($ K $) space (resp.\ CBB($ K $) space) if it satisfies the following for all pairwise distinct points $ p,x,y\in X $ with $ d(p,x)+d(x,y)+d(y,p) < 2D_K $.
	\begin{itemize}
		\item Let $ z $ be a point on a geodesic joining $ x $ to $ y $. The point $ \tilde{z}\in M_K $ with $ \tilde{d}^K(\tilde{z},\tilde{x}) = d(z,x) $ and $ \tilde{d}^K(\tilde{z},\tilde{y}) = d(z,y) $ satisfies
		\begin{equation*}
			d(p,z)\leq \tilde{d}^K(\tilde{p},\tilde{z}) \quad (\text{resp.\ } d(p,z)\geq \tilde{d}^K(\tilde{p},\tilde{z})).
		\end{equation*}
	\end{itemize}
	We say that a metric space is \textit{of curvature} $ \leq K $ (resp.\ $ \geq K $) if each point admits a neighborhood which is a CAT($ K $) space (resp.\ CBB($ K $) space).
\end{defn}
See \cite{AKP_Found} for equivalent and more general definitions.
\begin{defn}
	An \textit{Alexandrov space} is a complete length space that admits a curvature bound in the sense of Definition~\ref{def_Curv}.
\end{defn}
The notion of an Alexandrov space is not consistent within literature and the different definitions differ in whether they require completeness, local compactness, or a lower curvature bound.
Note that by the Globalization Theorem \cite[8.31]{AKP_Found} a geodesic Alexandrov space of curvature $ \geq K $ is also a CBB($ K $) space. That is, the local triangle comparison holds in fact globally. However, an analogous statement for upper curvature bounds is not true.

Alexandrov spaces of curvature $ \geq K $ (resp.\ $ \leq K $) generalize smooth complete connected Riemannian manifolds without boundary with sectional curvatures $ \geq K $ (resp.\ $ \leq K $) (cf.\ \cite[Theorem~ 12.2.2 + Exercise~12.8.4]{Pet} and \cite[ch.\ II, 1A.6]{BH}, respectively).\com{See cite{ABBc} for the case of nonempty boundary and upper curvature bound} When referring to Riemannian manifolds in the following, we always mean smooth complete connected Riemannian manifolds without boundary.
A natural generalization of strict inequalities for sectional curvatures in Riemannian manifolds is the following.
\begin{defn}\label{def_posCurv}
	Let $ \kappa\in\IR $ and $ X $ be a metric space such that for each point $ p $ there exists a constant $ \kappa_p>\kappa $ such that a neighborhood at $ p $ is of curvature $ \geq \kappa_p $. Then we call $ X $ of curvature $ >\kappa $.
\end{defn}
\com{This implies 4-point-comparison with strict $ < $. But it is not equivalent. Consider a surface with a single point of minimal curvature (e.g.\ rotation of $ x^4 $). Then all triangles ``experience also higher curvature'' and satisfy 4-point comparison always with strict inequality}
A proper Alexandrov space $ X $ of curvature $ \geq \kappa >0 $ is compact if it is not isometric to the real (half\=/)line (Lemma~\ref{lem_factsAS}~\ref{lem_diameterAS}+\ref{lem_OneDimAS}). Therefore, except for the real (half\=/)line, a locally compact Alexandrov space of positive curvature is of curvature $ \geq \kappa>0 $ if and only if it is compact.

\begin{defn}
	Given two nontrivial geodesics $ \gamma,\sigma $ in a metric space $ X $ issuing in a common point $ p\in X $, we define the \textit{angle}
	\begin{equation*}
		\angle_p(\gamma,\sigma):=\lim_{s,t\searrow 0} \tilde{\angle}^0_p(\gamma(s),\sigma(t))
	\end{equation*}
	if the limit exists. We define $ \Sigma'_p $ to be the set of all nontrivial unit-speed geodesics issuing in $ p $, pairwise identified if they have an angle of zero. The angle $ \angle_p $ defines a metric on $ \Sigma'_p $. The \textit{space of directions} $ \Sigma_p $ is the completion of $ \Sigma'_p $.
	
\end{defn}

We gather standard facts about Alexandrov spaces.
\begin{lem}\label{lem_factsAS}
	A complete, locally compact length space $ X $ has the following properties.
	\begin{enumerate}[label=(\roman*)]
		\item \label{lem_hopfRinow}\cite[ch.\ I, 3.7]{BH}. $ X $ is proper and geodesic.
		\item \label{lem_geodLimit}\cite[ch.\ I, 3.11]{BH}. A sequence of geodesics $ c_n:[0,1]\to X $ contained in a bounded ball admits a convergent subsequence that converges uniformly to a (possibly constant) geodesic $ c:[0,1]\to X $.
		\item If $ X $ is of curvature $ \geq \kappa $ (resp.\ $ \leq K $) and $ A \subseteq X $ is a closed and convex subset, then $ A $ is an Alexandrov space of curvature $ \geq \kappa $ (resp.\ $ \leq K $).
		\item \label{lem_angleMonotonicity}\cite[8.14+9.14]{AKP_Found}. If $ X $ is of curvature $ \geq \kappa $ (resp.\ $ \leq \kappa $) and $ \gamma,\sigma $ are nontrivial geodesics issuing in $ p\in X $, then the map $ (s,t)\mapsto \tilde{\angle}^\kappa_p(\gamma(s),\sigma(t)) $ is nonincreasing (resp.\ nondecreasing) in both arguments. In particular, all angles exist.
		\item \label{lem_lengthContLocGeod}\cite[9.45]{AKP_Found}. If $ X $ is of curvature $ \leq K $ and $ (\gamma_n:[0,1]\to X)_{n\in\IN} $ is a sequence of local geodesics which converge uniformly to some curve $ \gamma:[0,1]\to X $, then $ \gamma $ is a local geodesic and $ \lim\limits_{n\to\infty} L(\gamma_n) = L(\gamma). $
		\item \label{lem_locMinimizing}\cite[ch.\ II, 1.4]{BH}. If $ X $ is CAT($ K $), then any local geodesic of length $ <D_K $ is geodesic and each pair of points $ x,y\in X $ with $ d(x,y)<D_K $ is joined by a unique geodesic. This geodesic varies continuously with its endpoints.
		\item \label{lem_adjAngle}\cite[8.39]{AKP_Found}. If $ X $ is of curvature bounded below, $ c:[0,1]\to X $ is a geodesic joining distinct points $ x,y\in X $ and $ \gamma $ is a geodesic joining an inner point $ p\in c((0,1)) $ to some point $ q\in X\setminus \{p\}  $, then 
		\[\angle_p(\gamma',x') + \angle_p(\gamma',y')=\pi,\]
		where $ x',y',\gamma'\in\Sigma_p $ are the directions corresponding to $ -c,\ c $ and $ \gamma $, respectively.
		\item \label{lem_NonBranchingGeod}\cite[8.37]{AKP_Found}. If $ X $ is of curvature bounded below and two geodesics issuing in a common point coincide on an open interval, then one is an extension of the other.
		
		\item \label{lem_diameterAS}\cite[8.44]{AKP_Found}. If $ X $ is of curvature $ \geq \kappa $ and not one-dimensional, then $ \diam X \leq D_\kappa $.
		\item \label{lem_OneDimAS}\cite[15.18]{AKP_Found}. If $ X $ is one-dimensional and of curvature bounded below, then it is isometric to a connected complete one-dimensional Riemannian manifold with possibly nonempty boundary. 
	\end{enumerate}
\end{lem}
\begin{rem}
	Note that, other than for spaces of curvature $ \leq K $, for CAT($ K $) spaces Corollary~\ref{cor_injSoulPos} even holds without a lower curvature bound and is an immediate consequence of the CAT($ K $) property Lemma~\ref{lem_factsAS}~\ref{lem_locMinimizing}.\com{This is the special case of the result \cite[1.3]{AB_GaussEqInjRadSubspaces} where a lower bound for the injectivity radius of subspaces of extrinsic curvature bounded above was found.}
\end{rem}
The last statements refer to the dimension of an Alexandrov space which we introduce now.
\begin{defn}
	The \textit{dimension} of an Alexandrov space is defined to be its Hausdorff dimension.
\end{defn}
Notably, every finite-dimensional Alexandrov space is locally compact and the statements in Lemma~\ref{lem_factsAS} apply. The dimension of an Alexandrov space is an integer (or infinite). Furthermore, it agrees with the so-called strainer number which gives the construction of the distance charts, cf.\ \cite[ch.~10]{BBI}. We note a rather specific but simple observation of this theory for later use.
\begin{rem}\label{rem_dimOpen}
	Let $ X $ be an $ n $-dimensional Alexandrov space of curvature $ \geq \kappa $ and $ A\subseteq X $ be a closed and convex subset of dimension $ n $. Then there exists a point $ p\in A $ and a neighborhood in $ X $ of $ p $ which is contained in $ A $.
	
	This is a consequence of the distance coordinates in the following way. Since $ A $ is $ n $-dimensional, it contains an $ n $-strained point $ p $ which is clearly also an $ n $-strained point in $ X $. There exists the distance coordinate map $ f $ which is a homeomorphism from an open neighborhood $ U\subseteq X $ of $ p $ to some subset of $ \IR^n $ (see \cite[10.8.18]{BBI}). To prove that the distance coordinates map surjectively onto an open subset (see \cite[10.8.15]{BBI}) it suffices to consider points along geodesics from $ p $ to the strainer points $ (a_i,b_i) $ and their limits. Since $ A $ is convex and closed, they are contained in $ A $ and hence $ f|_{U\cap A} $ is already surjective onto $ f(U) $. But since $ f $ is injective, this implies that $ U\cap A = U $, that is, $ U $ is contained in $ A $.
\end{rem}
\begin{defn}\label{def_boundary}
	The \textit{boundary} of an $ n $-dimensional Alexandrov space of curvature bounded below is inductively defined to be the set of points $ p\in X $ where the space of directions $ \Sigma_p $, which is an Alexandrov space of curvature $ \geq 1 $ and dimension $ n-1 $ (\cite[7.10+7.11]{BGP}), has nonempty boundary. One-dimensional Alexandrov spaces are one-dimensional manifolds (Lemma~\ref{lem_factsAS}~\ref{lem_OneDimAS}) and their boundary is defined to be their topological boundary. Zero-dimensional Alexandrov spaces are singleton sets and have, by definition, empty boundary.
\end{defn}
The following statement is extracted from \cite{Mitsu} and reveals that the Alexandrov boundary behaves similarly to the manifold boundary.
\begin{lem}\label{lem_homology}
	Let $ X $ be an $ n $-dimensional Alexandrov space of curvature bounded below and $ U\subseteq X $ be a nonempty, open and connected subset. Then the $ n $-th singular homology group $ H_n(U;\IZ/2\IZ) $ in $ \IZ/2\IZ $ coordinates is nontrivial if and only if $ X $ is compact with empty boundary and $ U=X $.\com{If it is nontrivial it is equal to $ \IZ/2\IZ $. Probably easy to deduce from \cite{Mitsu} with algebraic topology. I did not check.}
\end{lem}
\begin{proof}
	First, note that if $ U\cap \partial X =\emptyset $, then $ U $ is an ``NB-space'' as defined in \cite{Mitsu} which allows us to apply their results.
	Let $ U=X $ be compact and $ \partial X = \emptyset $. In \cite{Mitsu} a notion of ``orientability'' is defined. They show that if $ X $ is ``orientable'', then $ H_n(X;\IZ)=\IZ $ (\cite[Theorem~1.8~(A)]{Mitsu}) and if it is not, then still $ H_n(X;\IZ/2\IZ)=\IZ/2\IZ $ (\cite[Corollary~5.7~(i)]{Mitsu}). We have in both cases $ H_n(X;\IZ/2\IZ)\neq \{0\} $ since $ H_n(X;\IZ)\neq \{0\} $ implies $ H_n(X;\IZ/2\IZ) \neq \{0\} $ which is immediate by the definition of singular homology or a trivial application of the universal coefficient theorem for homology\com{see Hatcher 3A.3}.
	To show the converse, assume $ U $ is a proper subset of $ X $ or $ X $ is not compact or $ X $ has nonempty boundary. If $ U\cap \partial X $ is nonempty, then, by \cite[Corollary~1.14]{Mitsu}, we obtain $ H_n(U;\IZ/2\IZ)=\{0\} $. If $ U\cap \partial X $ is empty, then either $ U=X $ and by assumption $ X $ is not compact or $ U\neq X $ in which case $ U $ is a nonempty proper open subset of the connected space $ X $ and thus not compact either. Hence, $ U $ is not compact but satisfies $ U\cap \partial X = \emptyset $ which allows us to apply \cite[Theorem~1.8~(C)]{Mitsu}.
\end{proof}
As a consequence, the boundary of Alexandrov spaces generalizes the manifold boundary in the following sense.
\begin{cor}\label{cor_boundaryManifold}
	If $ \varphi:X\to M $ is a homeomorphism between a finite-dimensional Alexandrov space $ X $ of curvature bounded below and a topological manifold $ M $, then it satisfies $ \varphi(\partial X)=\partial M $.
	\begin{proof}
		Let $ n:=\dim X $ and note that $ M $ must be $ n $-dimensional too. The relative homology groups of the space of directions at a point $ p\in X $ indicates if it has nonempty ($ H_{n-1}(\Sigma_p;\IZ/2\IZ)=\{0\} $) or empty ($ H_{n-1}(\Sigma_p;\IZ/2\IZ)\neq \{0\} $) boundary (see Lemma~\ref{lem_homology}). And by \cite[Theorem~6.4]{Mitsu} a sufficiently small ball $ B $ at $ p\in X $ is pointed homeomorphic to the tangent cone $ K(\Sigma_p) $ with the cone-tip $ 0_p $. The tangent cone is contractible but with the tip removed it is homotopy equivalent to $ \Sigma_p $. The long exact sequence for relative homology groups yields (all homology groups in $ \IZ/2\IZ $ coefficients)
		\[H_n(B,B\setminus\{p\})\approx H_n(K(\Sigma_p),K(\Sigma_p)\setminus\{0_p\}) \approx H_{n-1}(\Sigma_p). \]
		We can find an open neighborhood $ C\subseteq \varphi(B) $ of $ \varphi(p) $ which is homeomorphic to $ \IR^n $ or its closed half space. By excision, we conclude
		\begin{align*}
			H_{n-1}(\Sigma_p)&\approx H_n(B,B\setminus\{p\})
			\approx H_n(C,C\setminus\{\varphi(p)\})\\
			&\approx\begin{cases}
				\{0\} &\quad \text{ if $ \varphi(p) $ is a manifold boundary point},\\
				\IZ/2\IZ &\quad \text{ if $ \varphi(p) $ is a manifold inner point}.
			\end{cases}
		\end{align*}
		Therefore, $ p $ is a boundary point if and only if $ \varphi(p) $ is a (manifold) boundary point.
	\end{proof}
\end{cor}
A common construction that produces spaces with empty boundary is the following.
\begin{defn}\label{def_doubling}
	Let $ X $ be a finite-dimensional Alexandrov space of curvature $ \geq \kappa $ with nonempty boundary. The \textit{doubling} $ \overline{X} $ of $ X $ is defined to be the metric space obtained from gluing together two copies of $ X $ along its boundary equipped with the induced length metric.
\end{defn}
The doubling is again an Alexandrov space of curvature $ \geq \kappa $ and has empty boundary (cf.\ \cite[5.2]{P2} and \cite[2.1]{Pet_AppQ}). In particular, this implies the following.
\begin{lem}\label{lem_distBoundaryBound}
	Let $ X $ be an Alexandrov space of curvature $ \geq \kappa > 0  $ and finite dimension $ \geq 2 $ with nonempty boundary. Then the distance to the boundary $ \dist_{\partial X} $ is bounded by $ D_\kappa/2 $.
\end{lem}
\begin{proof}
	Consider the doubling $ \overline{X} $ of $ X $ which satisfies $ \diam \overline{X} \leq D_\kappa $ by Lemma~\ref{lem_factsAS}~\ref{lem_diameterAS}. Fix $ p\in X $ and assume $ \dist_{\partial X}(p)>0 $. Then the distance of $ p $ to its reflection point $ \bar{p} $ in the other copy of $ X $ satisfies $ d(p,\bar{p}) = 2 \inf_{q\in\partial X} d(p,q) $. Therefore,
	\[\dist_{\partial X}(p) = \frac12 d(p,\bar{p}) \leq \frac{D_\kappa}{2}. \qedhere \]
\end{proof}
\section{The Cheeger-Gromoll Soul}\label{sec_Soul}
In this section we summarize the results of \cite{P2} that generalize the soul to Alexandrov spaces. For convenience of the reader we do so in a detailed manner.
\subsection{The Soul Construction}
To define the soul let us first recall the definition of totally convex subsets.
\begin{defn}\label{def_totConv}
	Let $ X $ be a metric space. A subset $ A\subseteq X $ is said to be \textit{totally convex}\index{Convex!totally convex subset} if every pair of points in $ A $ can be joined by a local geodesic in $ X $ and the image of every such local geodesic is contained in $ A $.
\end{defn}
Total convexity is far more restricting than convexity. For instance, the only totally convex subset of the sphere $ \IS^n $ is $ \IS^n $ itself.
In Riemannian manifolds the soul of a manifold is a compact, totally geodesic and totally convex submanifold (without boundary). This leads to the following generalization to Alexandrov spaces.
\begin{defn}
	Let $ X $ be a finite-dimensional Alexandrov space of curvature bounded below. A nonempty, totally convex, compact subset of $ X $ without boundary is called a \textit{soul}\index{Soul} of $ X $.
\end{defn}
Note that, as a closed and convex subset, a soul is a finite-dimensional Alexandrov space with lower curvature bound itself and the boundary is defined. Since a closed and convex subset of a Riemannian manifold is a submanifold possibly with (non-smooth) boundary (\cite[1.6]{CG}), the preceding definition of a soul in Alexandrov space generalizes the Riemannian soul.

A compact space $ X $ without boundary always serves as its own soul and is thus rather uninteresting. But for other spaces the soul can serve as a compact and boundaryless reduction of the whole space, inheriting some of its topological and geometric properties. A natural question is if such a soul always exists. Cheeger and Gromoll gave a constructive proof for Riemannian manifolds with nonnegative sectional curvature in \cite{CG} which Perelman generalized to Alexandrov spaces in \cite[6]{P2}. To obtain totally convex compact subsets of a non\-compact space, the so-called Busemann functions will turn out to be useful.

\begin{defn}\label{def_busemannFunc}
	Let $ X $ be a metric space and $ c:[0,\infty)\to X $ be a unit-speed ray. Then the \textit{Busemann function}\index{Busemann function} $ b_c $ is given by
	$$ b_c:X\to\IR, \quad q\mapsto \lim_{t\to\infty}d(q,c(t))-t$$
\end{defn}
\begin{rem}\label{rem_BusemannWellDefined}
	By the standard and reverse triangle inequality, the map \linebreak$ {t\mapsto d(q,c(t)) - t} $ is non\-increasing and bounded from below. Hence, the Busemann function is well-defined. Furthermore, standard arguments show that it is continuous. Superlevel sets can be understood as \textit{half spaces}, e.g.\ $ X\setminus b_c^{-1}([0,\infty)) = \cup_{r>0}B_r(c(r)) $.
\end{rem}
The following establishes the existence of rays in noncompact spaces to enable the use of Busemann functions. This is a well-known and often used fact and mainly due to the compactness of the space of directions which was proved in \cite[7.3]{BGP}. We include a proof for the convenience of the reader.
\begin{lem}\label{lem_rayExistence}
	Let $ X $ be a non\-compact finite-dimensional Alexandrov space of curvature bounded below. Then for every point $ p\in X $, there exists a ray issuing in $ p $.
\end{lem}
\begin{proof}
	Since $ X $ is proper and not compact it must be unbounded. Therefore, there exists a sequence of points $ (q_i)_{i\in\IN} $ with $ d(p,q_i)\to \infty $. For $ i\in\IN $, fix the directions $ \xi_i\in\Sigma_p $ corresponding to a geodesic from $ p $ to $ q_i $. The space of directions $ \Sigma_p $ is compact and hence there is a convergent subsequence of directions $ (\xi_{i_n})_{n\in\IN} $ which correspond to unit-speed geodesics $ (c_n)_{n\in\IN} $ from $ p $ to $ q_{i_n} $. Let $ \xi\in\Sigma_p $ be the limit of that sequence and rename $ \xi_n:=\xi_{i_n} $. Fix $ t\in[0,\infty) $ and consider the closed ball of radius $ t $ at $ p $ which contains the restrictions $ (c_n|_{[0,t]})_{n\in\IN} $ for sufficiently large $ n $. By Lemma~\ref{lem_factsAS}~\ref{lem_geodLimit}, there exists a subsequence $ (c_{n_k}|_{[0,t]})_{k\in\IN} $ which converges to a unit-speed geodesic $ r_t $ of length $ t $. By the law of cosines, it has direction $ \xi $\com{Consider the isosceles triangles $ \tilde{\Delta}^K(p,r_t(s),c_n(s)) $ for fixed $ s $ where the side tends to zero and hence the opposite angle. With triangle inequality we obtain $ \angle_p(r_t,c)=0 $}. We show that the whole sequence $c_n|_{[0,t]} $ converges to $ r_t $. Consider $ s\in [0,t] $ and the comparison triangle $ \tilde{\Delta}^K(p,c_n(s),r_t(s)) $ with the angle 
	$$ \tilde{\alpha}_n(s):=\tilde{\angle}_p^K(c_n(s),r_t(s))\leq \angle_p(\xi_n,\xi) $$
	This gives a bound independent of $ s $ that tends to zero. Hence, so does the length $ d(c_n(s),r_t(s)) $ of the side in the isosceles triangle $ \tilde{\Delta}^K(p,c_n(s),r_t(s)) $ opposite to the angle $ \tilde{\alpha}_n(s) $. Consequently, the geodesics $ c_n|_{[0,t]} $ converge uniformly to $ r_t $. 
	This allows us to define, for all $ t\in[0,\infty) $, the pointwise limit $ r(t):=\lim_{n\to\infty}c_n(t)=r_t(t) $. To show that $ r $ is indeed a ray, let $ t\in[0,\infty) $. Since $ (c_n|_{[0,t]})_{n\in\IN} $ converges uniformly to $ r_t $, we have $ r(s)=r_t(s) $ for all $ s\in[0,t] $ and therefore $ r|_{[0,t]}=r_t $ is a geodesic. Therefore, $ r $ is a ray.
\end{proof}
We now give a short informal description of the soul construction. To find a soul of a finite-dimensional Alexandrov space $ X $ of nonnegative curvature, proceed as follows. If $ X $ is not compact, choose an arbitrary point $ p_0\in X $. Obtain a compact totally convex subset as the intersection of superlevel sets of all Busemann functions corresponding to rays emanating from $ p $. As long as this set has non\-empty boundary, obtain the next set as the subset of maximum distance to the boundary. In each step the dimension is reduced so that we arrive at a subset of dimension zero or a subset without boundary after finitely many steps. This is a soul (see Figure~\ref{fig_soulCyl}).
\begin{figure}[h]
	\centering
	\includegraphics[width=0.8\linewidth]{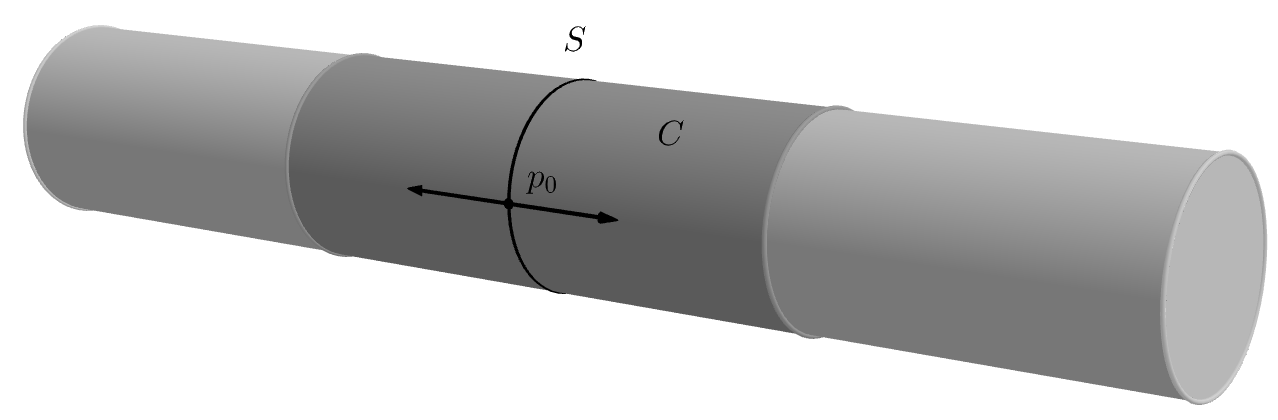}
	\caption{The soul construction on the cylinder terminating in a soul $ S $. $ C $ is a superlevel set of the minimum of the Busemann functions corresponding to the two rays issuing in $ p_0 $}\label{fig_soulCyl}
\end{figure}

We need to prove that the construction outlined above is well-defined, of which one part is that it yields a totally convex subset in each step. For this, the concavity of the Busemann function and the boundary distance function is crucial in the following way.
\begin{defn}\label{def_concaveFunc}
	Let $ X $ be a metric space. A function $ f:X\to \IR $ is said to be \textit{(strictly) concave}\index{Concave function}\index{Convex!function} if, for every nontrivial geodesic $ c:[0,L]\to X $, the real map $ f\circ c $ is (strictly) concave.
\end{defn}
Note that some authors use the notion of ``convexity'' of functions (e.g.\ \cite[1.9]{CG},\cite[4.2]{ShioCompNoncomp},\cite[6.1]{P2}) which corresponds to our definition of concavity.
Also the meaning of ``strictly concave'' varies in the literature (see Remark~\ref{rem_strictlyConcave}).
Note that even the composition of a concave function with a \textit{local} geodesic is concave since concavity of real continuous functions is a local property (e.g.\ see \cite[Theorem~1.1.3]{NiPe}).
An immediate consequence is the following.\com{Note that concavity does not imply continuity for functions on $ X $. See \cite{Mash}?}
\begin{lem}\label{lem_concaveSuperLevelSets}
	Let $ X $ be a metric space and $ f:X\to\IR $ be a concave function. Then for each $ a\in \IR $ the superlevel set $ f^{-1}([a,\infty)) $ is totally convex.
\end{lem}
Consider a geodesic segment $ c $ on the southern hemisphere of the two-dimensional sphere of curvature $ \kappa>0 $ with north pole $ N $. Then the distance function $ t\mapsto d(c(t),N) $ is concave. By triangle comparison (Definition~\ref{def_Curv}) one can translate this to Alexandrov spaces of curvature $ \geq \kappa >0 $, i.e.\ the distance function is concave above $ \frac{D_\kappa}{2} $, half the diameter of the sphere. In the case of non\-negative curvature where $ D_\kappa=\infty $ the distance function indeed becomes concave ``at infinity'', that is, the Busemann function is concave. This is well known and a proof can be found for instance in \cite[4.2]{ShioCompNoncomp} or \cite[Exercise~8.25]{AKP_Found}.\com{They conclude with $ d=\dist_{c(t)} $ from $ d'^2 + dd'' = (d^2/2)''\leq 1 $ that $ d''\leq \frac{1-d'^2}{d} = \frac{1-\cos^2\alpha}{d}\leq \frac{1}{d} $ and hence for each $ \eps>0 $ there is large $ t $ such that $ d $ is $ \eps $-concave. We used the chain rule for the derivative in the barrier sense (reference?). Hence for all $ \eps>0 $ the function $ (b_c\circ\gamma-\eps t^2) $ is concave.}
\begin{lem}\label{concaveBusemann}\index{Busemann function}
	Let $ X $ be an Alexandrov space of non\-negative curvature and $ c:[0,\infty)\to X $ be a unit-speed ray. Then the Busemann function $ b_c $ is concave.
\end{lem}
The concavity of the boundary distance function 
\[\dist_{\partial X} : X \to [0,\infty),\ p\mapsto d(p,\partial X).\]
is less obvious. It was proved in \cite{P2} and later put into a more general context in \cite{ABConvex}. In \cite[3.3.1]{PetrGradFlow} a proof can be found that simplifies the application of some of the structure theory in \cite{P2} to the use of the gradient exponential map.
To treat the case of a positive lower curvature bound, we introduce a generalized notion of concavity.
\begin{defn}
	Let $ c\in \IR $, $ I\subseteq \IR $ an interval, $ t_0\in I $ and $ h:I\subseteq \IR \to \IR $ be a continuous function. We say that $ h $ satisfies $ h''(t_0)\leq c $ \textit{in the barrier sense} if there exists a smooth function $ g:I\to \IR $ such that $ h\leq g $, $ h(t_0)=g(t_0) $ and $ g''(t_0)\leq c $. For a continuous function $ \varphi:\IR\to\IR $, we say that $ h''\leq \varphi\circ h $ \textit{in the barrier sense} (or $ h $ is \textit{$ (\varphi \circ h) $\=/concave}) if for all $ t_0\in I $, we have $ h''(t_0)\leq \varphi(h(t_0)) $ in the barrier sense. A function $ f:X\to \IR $ on a metric space $ X $ is called \textit{$ (\varphi \circ f) $\=/concave} if $ h:=f\circ c $ is $ (\varphi\circ h) $\=/concave for each unit-speed geodesic $ c $ in $ X $. 
\end{defn}
\com{See my real concave functions notes}Note that there are other common definitions of differential inequalities, some of which are equivalent to the above. We are especially interested in the case $ \varphi(f)=\lambda-\kappa f $ with $ \lambda,\kappa\in\IR $ for which equivalent definitions can be found in \cite[3.14]{AKP_Found}.
\begin{prop}\label{prop_distConcave}
	Let $ X $ be a finite-dimensional Alexandrov space of curvature $ \geq \kappa $ with non\-empty boundary $ \partial X $.
	\begin{enumerate}[label=(\roman*)]
		\item \cite[6.1]{P2}. If $ \kappa=0 $, then $ \dist_{\partial X} $ is concave.
		\item \cite[1.1~(5B)]{ABConvex}. If $ \kappa>0 $, then $ f:=\sin\left(\sqrt{\kappa}\dist_{\partial X}\right) $ is $ (-\kappa f) $\=/concave. \com{Note that $ \dist_{\partial X}(X\setminus \partial X)\in(0,\frac{D_\kappa}{2}] $. This comes from the fact that the doubling of $ X $ has diameter $ \leq D_\kappa $ and hence $ d(p,\partial X) = \frac12 d(p,\bar{p})\leq \diam \bar{X} $}
	\end{enumerate}
\end{prop}
\begin{rem}\label{rem_strictlyConcave}
	For $ \kappa>0 $ the function $ \sin(\sqrt{\kappa}\dist_{\partial X}) $ is, in particular, strictly concave on $ X\setminus \partial X $ but $ \dist_{\partial X} $ itself is not. Along any geodesic $ c $ that minimizes the distance of some point $ p\in X\setminus \partial X $ to the boundary, the composition $ \dist_{\partial X} \circ~c $ is linear. For illustration, take a geodesic connecting the north pole to the equator on the closed northern hemisphere of $ \IS^2 $. Note that in \cite[6.1]{P2} the boundary distance function $ \dist_{\partial X} $ is called ``strictly convex'' (or ``strictly concave'' when cited in \cite{ABConvex} or \cite{W}) which, hence, does not correspond to our notion of a strictly concave function. 
\end{rem}
We retain the following as a direct consequence of Proposition~\ref{prop_distConcave}.
\begin{cor}\label{cor_distBoundUniqMax}
	Let $ X $ be a finite-dimensional Alexandrov space of curvature $ \geq \kappa > 0 $ with non\-empty boundary and $ c:[0,L]\to X\setminus \partial X $ a nontrivial geodesic, then $ \dist_{\partial X}\circ~c  $ attains a unique maximum.
\end{cor}
\begin{proof}
	This is obvious for zero- and one-dimensional Alexandrov spaces (see Lemma~\ref{lem_factsAS}~\ref{lem_OneDimAS}). Assume $ \dim X \geq 2 $ and set $ h:=\dist_{\partial X}\circ~c $ and $ f:=\sin(\sqrt{\kappa}~h) $. The distance $ h $ to the closed set $ \partial X $ is positive and is bounded by $ D_\kappa/2 $ (see Lemma~\ref{lem_distBoundaryBound}). Hence, we have $ m:=\min f >0 $ and, by Proposition~\ref{prop_distConcave}~(ii),
	\begin{equation}\label{eq_fkconvex}
		f''\leq -\kappa f \leq -\kappa m \quad\text{ (in the barrier sense)}
	\end{equation} 
	with $ -\kappa m<0 $.
	Now, assume $ h $ attains its maximum at two distinct points. Then by concavity of $ h $ (Proposition~\ref{prop_distConcave}~(i)), it is constant between those two points. But then $ f $ is also constant on that subinterval which is a contradiction to \eqref{eq_fkconvex}. 
\end{proof}
We summarize the properties of the superlevel sets of the soul construction (see also Figure~\ref{fig_soulSuperlevelSets}). 
\begin{lem}\label{lem_soulSuperlevelSets}
	Let $ X $ be an $ n $-dimensional Alexandrov space of non\-negative curvature. If $ X $ is not compact (resp.\ compact with non\-empty boundary), then there exists $ a\leq 0 $ and an interval $ I=[a,\infty) $ (resp.\ $ I=[a,0] $) and a family $ \{C_t:t\in I \} $ of non\-empty, compact and totally convex subsets $ C_t\subseteq X $ which satisfy:
	\begin{enumerate}[label=(\roman*)]
		\item If $ t_1,t_2\in I $ with $ t_1\leq t_2  $, then $$ C_{t_1} = \{p\in C_{t_2}:d(p,\partial C_{t_2})\geq t_2-t_1\}, $$ in particular $ C_{t_1}\subseteq C_{t_2} $.
		\item $ \bigcup\limits_{t\in I} C_t = X $.
		\item \label{lem_CtCont}The map $ t\mapsto C_t $ is continuous w.r.t.\ the Hausdorff distance.
		\item $ \dim C_{a} < \dim X $.
	\end{enumerate} 
\end{lem}
\begin{figure}[h]
	\centering
	\includegraphics[width=0.7\linewidth]{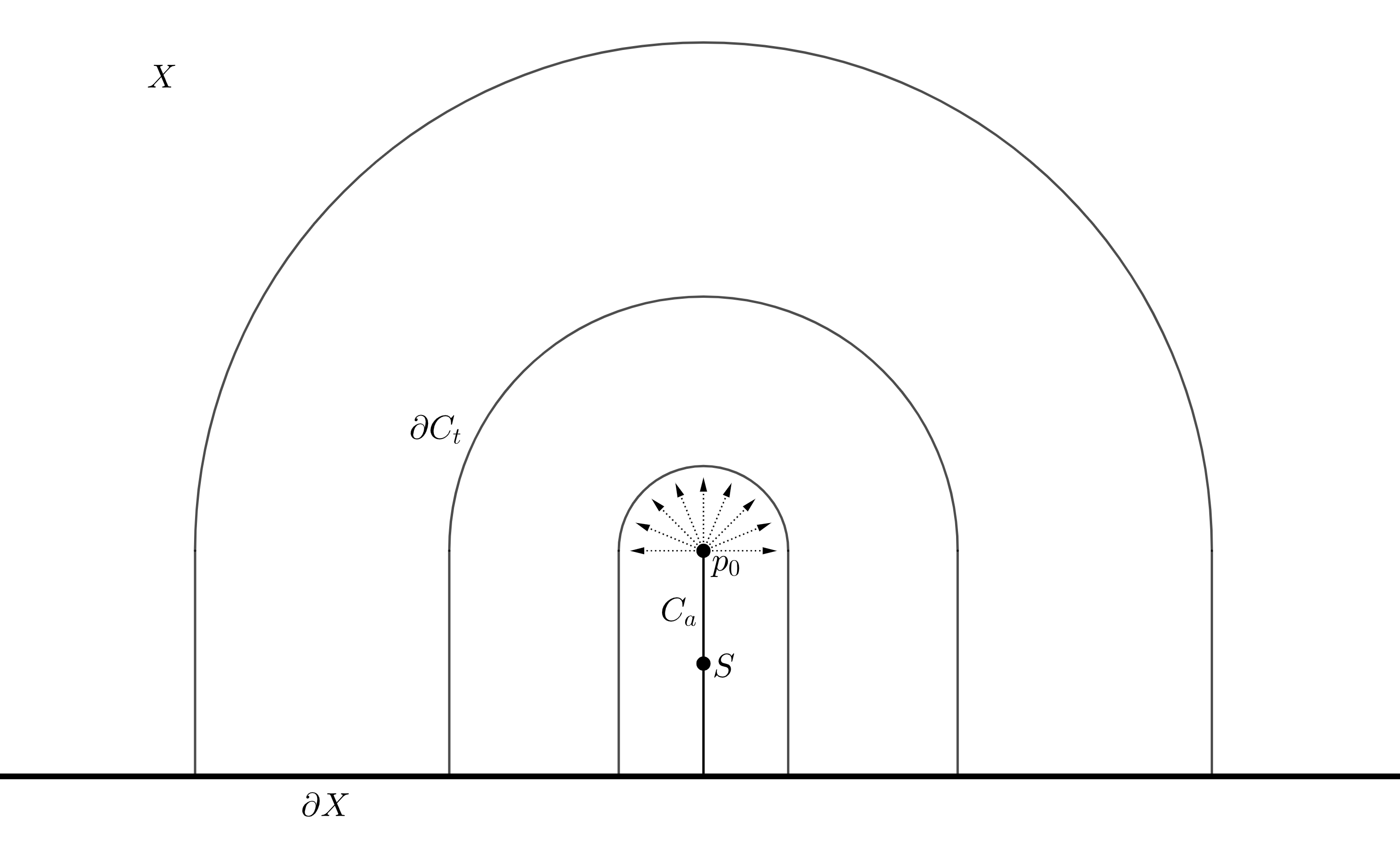}
	\caption{The soul $ S $ and the sets $ C_t $ from Lemma~\ref{lem_soulSuperlevelSets} for a closed half-plane $ X $ when starting in $ p_0 $. Some directions of rays issuing in $ p_0 $ are indicated}\label{fig_soulSuperlevelSets}
\end{figure}
\begin{proof}
	The following proofs are largely analogous to the Riemannian case and of elementary character. The description in (i) was already formulated in \cite{CG}. Continuity of the map $ t\mapsto C_t $ in (iii) was treated in \cite[\S 2]{S}.
	
	We assume that $ X $ is not compact. The compact case follows mostly analogously and is sometimes simpler with respective differences addressed in brackets.
	Fix some point $ p_0\in X $. Then there exists at least one ray emanating from $ p_0 $ by Lemma~\ref{lem_rayExistence}. We consider the set $ R_{p_0} $ of all such unit-speed rays and define the infimum over all the corresponding Busemann-functions 
	\begin{align}\label{eq_busemannDef}
		b&:X\to\IR,\ p\mapsto \inf_{c\in R_{p_0}} b_c(p) = \inf_{c\in R_{p_0}} \lim_{s\to\infty} d(p,c(s))-s.
	\end{align}
	Estimating $ b(p)\geq -d(p,p_0) $, we see that $ b $ is well-defined. 
	Now set $ f:=b $ (resp.\ $ f:=\dist_{\partial X} $) and define for $ t\in \IR $ the superlevel set
	\begin{align*} 
		C_t&:=f^{-1}([-t,\infty)).
	\end{align*}  
	By Proposition~\ref{concaveBusemann} (resp.\ Proposition~\ref{prop_distConcave}), the function $ f $ is an infimum over concave functions and thus itself concave. Fix $ t\in\IR $. By concavity of $ f $, the superlevel set $ C_t $ is totally convex. If $ X $ is compact, then the closed subset $ C_t $ is clearly compact itself. Otherwise, assume $ C_t $ would not be compact for some $ t\in\IR $. Then Lemma~\ref{lem_rayExistence} gives a unit-speed ray $ c $ emanating from $ p_0 $ with image in $ C_t $. But $ b(c(t+1)) \leq b_c(c(t+1))=-t-1<-t $ and hence $ c(t+1) $ cannot be contained in $ C_t $, a contradiction. We conclude that $ C_t $ is compact.	
	The set $ C_t $ could be empty but $ f $ is continuous and has compact superlevel sets and therefore assumes a maximal value $ -a $. This allows us to set $ I:=[a,\infty) $ (resp.\ $ I:=[a,0] $) and obtain the family $ (C_t)_{t\in I} $ of non\-empty, compact and totally convex subsets.

	(i) We present the proof for noncompact $ X $. The compact case follows similarly where instead of points $ c(s) $ on rays we directly consider closest points on the boundary. 
	We first show 
	\begin{equation}\label{eq_levelsetBoundary}
		f^{-1}(\{-t\}) = \partial C_t.
	\end{equation}
	By continuity of $ f $, we have $ \partial C_t\subseteq f^{-1}(\{-t\}) $. To show the converse inclusion let $ p\in f^{-1}(\{-t\}) $. Assume $ p $ is an inner point of $ C_t $, then there exists $ \varepsilon>0 $ such that $ B_\varepsilon(p)\subseteq C_t $. The equality $ f(p)=-t $ allows us to choose a unit-speed ray $ c $ issuing in $ p_0 $ and $ s>0 $ such that $ d(p,c(s))-s\leq -t + \frac{\varepsilon}{4} $. Then there is a geodesic from $ p $ to $ c(s) $ and we consider the point $ q $ on this geodesic with distance $ \frac{\varepsilon}{2} $ to $ p $. Recall that $ d(q,c(s))-s $ is non\-increasing in $ s $ and thus not smaller than $ b_c(q) $. We then have 
	\begin{align*}
		f(q)\leq b_c(q)&\leq d(q,c(s))-s\\
		&=d(p,c(s))-s-\frac{\varepsilon}{2}\\
		&\leq -t - \frac{\varepsilon}{4} < -t,
	\end{align*}
	a contradiction to $ q\in C_t $. Therefore, \eqref{eq_levelsetBoundary} follows.
	
	Now let $ t_1,t_2\in I $ with $ t_1\leq t_2 $ and $ p\in C_{t_1} $. We show that $ d(p,\partial C_{t_2}) \geq t_2-t_1 $. The reverse triangle inequality yields for all $ r\in\partial C_{t_2} $, all rays $ c\in R_{p_0} $ and all $ s>0 $ the inequality
	\begin{align*}
		d(p,r) &\geq d(p,c(s)) - d(c(s),r)\\
		&=d(p,c(s)) - s -  (d(c(s),r)-s)\\
		&\geq b_c(p) - (d(c(s),r)-s)\\
		&\geq -t_1 - (d(c(s),r)-s),
	\end{align*} 
	and hence $ d(p,r)\geq -t_1-f(r) = t_2-t_1 $. Since this holds for all $ r\in\partial C_{t_2} $, we obtain $ d(p,\partial C_{t_2})\geq t_2-t_1 $. Let conversely $ p\in C_{t_2} $ with $ d(p,\partial C_{t_2})\geq t_2-t_1 $. Let $ c\in R_{p_0} $, $ s\geq\max\{0,-t_2\} $ and consider a geodesic from $ p $ to $ c(s+t_2) $. Since $ b(c(s+t_2))\leq b_c(c(s+t_2))=-s-t_2\leq -t_2 $, this geodesic must cross $ \partial C_{t_2} $ at some point. Let $ r_s $ denote that crossing point. For all sufficiently large $ s $,
	\begin{align*}
		d(p,c(s+t_2))-(s+t_2) &= d(p,r_s) + d(r_s,c(s+t_2)) - (s+t_2)\\
		&\geq d(p,\partial C_{t_2}) + b_c(r_s)\\
		&\geq t_2-t_1 + f(r_s)\\
		& = -t_1,
	\end{align*} 
	from which we obtain $ b_c(p)\geq -t_1 $ for all $ c\in R_{p_0} $. Consequently, $ f(p)\geq -t_1 $ and in particular $ p\in C_{t_1} $.
	
	(ii) This follows immediately from $ \bigcup\limits_{t\in I} C_t = f^{-1}(\IR) = X $.
	
	(iii) Fix $ t\in I $ and $ \eps>0 $. Let us denote the $ \eps $-neighborhood of a subset $ A\subseteq X $ by $ U_\eps(A):=\{p\in X: d(p,A)<\eps\} $. We have to show that there exists $ \delta>0 $ such that $ C_{t+\delta}\subseteq U_{\eps}(C_t) $ (if $ I=[0,\infty) $ or $ t<0 $) and $ C_t\subseteq U_{\eps}(C_{t-\delta}) $ (if $ t>a $).\\
	The former relation follows just from continuity of $ f $. To see this, assume the contrary. Then there exists a sequence $ (p_n)_{n\in \IN} $ such that $ p_n\in C_{t+1/n} $ and $ d(p_n,C_t)\geq \eps $ (if $ I=[a,0] $, we may start the sequence at $ n $ large enough such that $ t+\frac1n<0 $). By compactness of $ C_{t+1} $ (resp.\ $ X $), there exists a limit $ p $ of a convergent subsequence $ (p_{n_k})_{k\in\IN} $. We have $ f(p_n)\geq -t-\frac1n $ and, by continuity, $ f(p)\geq -t $, i.e.\ $ p\in C_t $. A contradiction to $ d(p_{n_k},C_t)\geq \eps $. 
	
	To prove the second relation, assume on the contrary that for all $ \delta>0 $ the subset $ C_t\setminus U_\eps(C_{t-\delta}) $ is nonempty. Let $ p $ be a limit point of a sequence of points $ p_n\in C_t\setminus U_\eps(C_{t-\frac1n})  $ (start sequence at sufficiently large $ n $ such that $ t-\frac1n >a $). Continuity of $ f $ and the relation $ d(p,C_{t-\delta})\geq \eps $ for all $ \delta>0 $ imply $ p\in\partial C_t $. Consider a unit-speed geodesic $ c:[0,L]\to X $ joining $ p $ to some point in $ C_a $. Using \[f(c(0)) = -t < -a = f(c(L))\] and concavity of $ f\circ c $, we obtain for all $ s\in(0,L] $ that $ f(c(s))>-t $, that is, $ \delta_s:=d(c(s),\partial C_t))>0 $ and $ c(s)\in C_{t-\delta_s} $. For $ s\in(0,\eps) $ this yields a contradiction with $ s = d(p,c(s)) \geq d(p,C_{t-\delta_s})\geq \eps. $
		
	(iv) Assume that $ X $ is compact. As a closed and convex subset of $ X $, the set $ C_a $ is again a finite-dimensional Alexandrov space of non\-negative curvature. Assume that $ C_{a} $ has dimension $ n:=\dim X $. Then by Remark~\ref{rem_dimOpen}, there exists an open neighborhood $ U $ in $ X $ of some point $ p\in C_a $ which is contained in $ C_{a} $. Since $ \partial X $ is closed (\cite[4.6]{P2}), we can consider the shortest geodesic $ c $ from $ p $ to $ \partial X $. Starting at the inner point $ p $, this geodesic must pass through some ball $ B_\varepsilon(p)\subseteq U \subseteq C_{a}  $ with $ \varepsilon>0 $. Hence, the distance function $ t\mapsto d(c(t),\partial X) $ is constant with value $ -a $ on $ [0,\varepsilon) $ but it should decrease by the choice of $ c $, a contradiction. We conclude that $ \dim C_{a}<\dim X $. If $ X $ is not compact, let $ t>a $. Setting $ X':=C_t $, we apply the compact version of this Proposition. Then (i) implies that the obtained maximum level set $ C'_{a'} $ coincides with $ C_{a} $ and the previous argument ensures $ \dim C_{a}<\dim C_{t}\leq\dim X. $
\end{proof}

\begin{prop}\label{prop_soul}
	\text{\normalfont(\cite[6.2]{P2})}. Every finite-dimensional Alexandrov space of non\-negative curvature contains a soul.
\end{prop}
\begin{proof}
	Let $ X $ be a finite-dimensional Alexandrov space of non\-negative curvature. If $ X $ is not compact, we apply Lemma~\ref{lem_soulSuperlevelSets} to get a non\-empty, compact and totally convex subset $ X_1:=C_{a} $, otherwise we set $ X_1:=X $.
	Iteratively, we set $ X_{n+1} $ to be the set $ C_{a} $ obtained from applying Lemma~\ref{lem_soulSuperlevelSets} to the compact space $ X_n $ as long as $ X_{n} $ has boundary and $ \dim X_n\geq 1 $. The obtained set $ X_{n+1} $ is again a non\-empty, compact and totally convex subset of $ X_n $. This is again an Alexandrov space of non\-negative curvature and qualifies for further application of Lemma~\ref{lem_soulSuperlevelSets}. 
	By transitivity of total convexity, $ X_{n+1} $ is a totally convex subset not only of $ X_n $ but also of $ X $. 
	Furthermore, we have $ \dim X_{n+1}<\dim X_n $ and thus obtain after a finite number of steps either a totally convex singleton or a non\-empty, compact and totally convex subset without boundary. In both cases we obtained a soul of $ X $.
\end{proof}
\begin{defn}
	We call a soul obtained by the preceding construction a \textit{Cheeger-Gromoll soul}.
\end{defn}
The soul of a positively curved space is given by the following.
\begin{prop}[\text{\cite[6.2]{P2}}]\label{prop_soulPoint}
	Let $ X $ be a finite-dimensional Alexandrov space of positive curvature which is noncompact or has nonempty boundary. Then each Cheeger-Gromoll soul of $ X $ consists of a single point.
\end{prop}
\begin{proof}
	Let $ S $ be a Cheeger-Gromoll soul and let $ (C_t)_{t\in I} $ be the superlevel sets from Lemma~\ref{lem_soulSuperlevelSets} that led to $ S $. By property (i), we can assume that $ X $ is compact (if not, set $ X:=C_t $ with $ t>a $) and thus satisfies a positive lower curvature bound. 
	Suppose the set $ C_a $ of maximum distance to the boundary contains two distinct points. By convexity, it must contain a geodesic joining those two points. But, by Corollary~\ref{cor_distBoundUniqMax}, $ \dist_{\partial X} $ assumes a unique maximum on this geodesic, a contradiction. Therefore, $ S=C_a $ consists of a single point.
\end{proof}
Recall that the Soul Theorem \cite[Theorem~2.2]{CG} states that a complete noncompact Riemannian manifold of nonnegative curvature is diffeomorphic to the normal bundle of its soul. 
This does not generalize to Alexandrov spaces in the following sense.\hspace*{7cm}
\begin{wrapfigure}{r}{0.29\textwidth}
	\vspace*{0.7cm}
	\centering
	\includegraphics[width=0.8\linewidth]{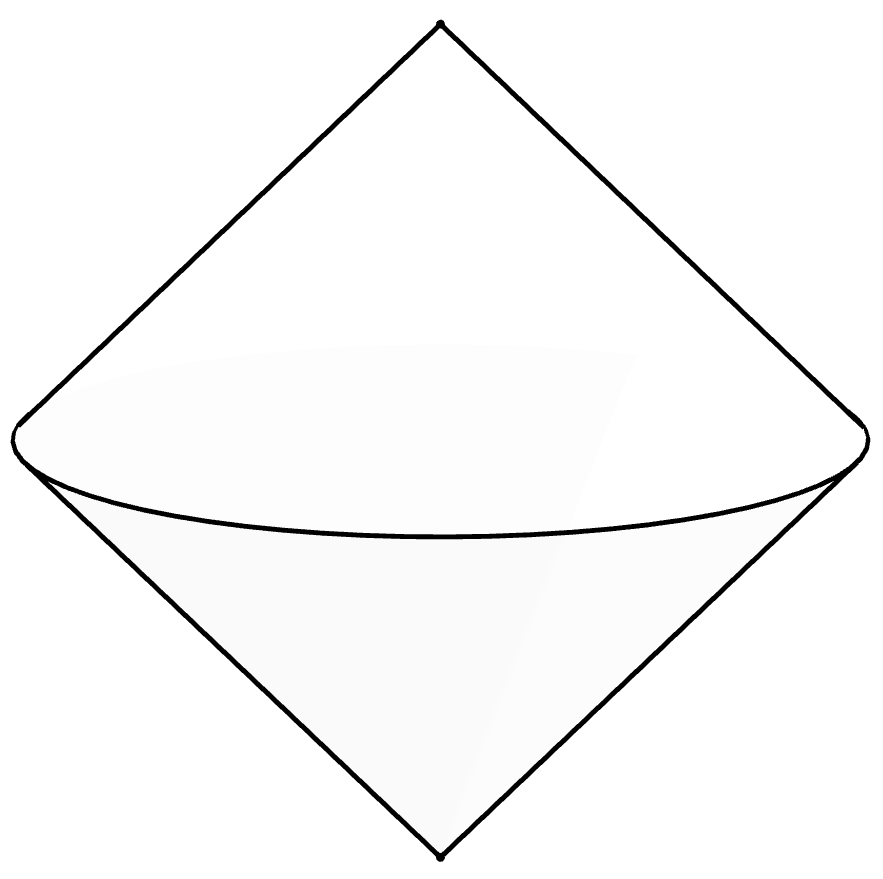}
	\caption{The soul $ S $ in Example~\ref{ex_soulTheoremCounterexample}}\label{fig_exPerelmanSoulDoublingCone}
	\vspace*{-0.2cm}
\end{wrapfigure}
\begin{ex}[\text{\cite[6.4]{P2}}]\label{ex_soulTheoremCounterexample}
	There is a finite-dimensional Alexandrov space of nonnegative curvature without boundary which is not homeomorphic to the total space of a locally trivial fiber bundle over its soul.
	We present the example \cite[6.4]{P2} with real instead of complex projective spaces. Consider $X:=\IR^3/{\sim}  $ where $ x\sim \pm x $ with induced length metric. $ X $ is isometric to the metric cone $ K(\IR P^2) $ over the real projective plane.
	Set $ Y:=\{(x,y,z)\in X: x^2+y^2\leq 1\} $ and consider the doubling $ \overline{Y} $ (see Definition~\ref{def_doubling}). Starting the soul construction in the origin yields the Cheeger-Gromoll soul $ S $ that is the doubling of $ \{(x,y,z)\in X: x^2+y^2\leq 1, z=0\}\subseteq X $ and hence homeomorphic to a sphere $ \IS^2 $ (in fact isometric to the doubling of the closed unit ball in the cone $ K(\IR P^1) $, see Figure~\ref{fig_exPerelmanSoulDoublingCone}). 
	Suppose there exists a projection $ \pi:\overline{Y}\to S $ defining a locally trivial fiber bundle. Then there exist disjoint open balls $ U,V\subseteq S $ with $ \pi(0)\in U $ such that neither $ U $ or $ V $ contain the image $ \pi(\overline{0}) $ of the reflection $ \overline{0} $ of the origin and the open subsets $ \pi^{-1}(U) , \pi^{-1}(V)\subseteq X $ are homeomorphic to the product of some fiber with $ U $ (resp.\ $ V $). Therefore, there exists a homeomorphism $ \varphi: \pi^{-1}(U) \to \pi^{-1}(V)$. It induces an isomorphism between the relative singular homology groups $ H_3(\pi^{-1}(U),\pi^{-1}(U)\setminus \{0\}) $ and $ H_3(\pi^{-1}(V),\pi^{-1}(V)\setminus \{\varphi(0)\}) $ (with $ \IZ $ coefficients). The following computation shows that there are indeed not isomorphic, yielding the desired contradiction.
	By excision, we can replace the open subsets by open balls $ B_U\subseteq \pi^{-1}(U) $ and $ B_V\subseteq \pi^{-1}(V) $ at $ 0 $ and $ \varphi(0) $, respectively. Note that $ B_U $ is contractible and $ B_U\setminus\{0\} $ is homotopy equivalent to $ \partial B_U\approx \IR P^2 $. Similarly, $ B_V $ is contractible, too, but $ B_V\setminus\{\varphi(0)\} $ is homotopy equivalent to $ \partial B_V \approx \IS^2 $. The corresponding long exact homology sequences yield the remaining isomorphisms
	\begin{align*}
		H_3(\pi^{-1}(U),\pi^{-1}(U)\setminus \{0\}) &\approx H_3(B_U,B_U\setminus\{0\}) \approx H_2(\IR P^2) \approx \IZ/2\IZ, \\
		H_3(\pi^{-1}(V),\pi^{-1}(V)\setminus \{\varphi(0)\}) &\approx H_3(B_V,B_V\setminus \{\varphi(0)\}) \approx H_2(\IS^2) \approx \IZ.
	\end{align*}
	Another formulation of the above is the following. Unlike for Riemannian manifolds, the (local) topology of an Alexandrov space is not determined by its soul. The  manifold $ S\times \IR $ has the same (isometric) Cheeger-Gromoll soul $ S $ as the ``non-manifold'' $ \overline{Y} $. That $ \overline{Y} $ is not a manifold can be proved by a computation of relative homology groups analogous to the above.
\end{ex}
Nonetheless, a weaker version of the Soul Theorem generalizes to Alexandrov spaces. Namely, the soul is a strong deformation retract of the whole space. This is described in the next section.
\subsection{\v{S}arafutdinov's Retraction}
In \cite{S} a non-expanding retraction of a Riemannian manifold onto its soul was found. It was generalized to Alexandrov spaces in \cite[6.3]{P2} and sparked a subsequent study of gradient flows in Alexandrov spaces which give a modern formulation of this retraction. We refer to \cite{PP,PetrGradFlow} and \cite[ch.\ 16]{AKP_Found}.
\begin{prop}[\text{\cite[6.3]{P2}} \v{S}arafutdinov's Retraction]\label{prop_sarafutdoinovsRetraction}
	Let $ X $ be a finite-dimensional Alexandrov space of non\-negative curvature and $ S $ be a Cheeger-Gromoll soul of $ X $. Then there exists a continuous map $ F:X\times [0,1]\to X $ that satisfies for all $ t\in [0,1],\ x\in X,\ s\in S $, 
	\[ F(x,0)=x,\quad F(s,t)=s,\quad F(x,1)\in S, \] 
	that is, $ S $ is a strong deformation retract of $ X $. Furthermore, $ F $ is non-expanding, that is, it satisfies for all $ p,q\in X $ and $ 0\leq t_1 \leq t_2 \leq 1 $ the inequality
\begin{align*}
	d(F(p,t_2),F(q,t_2)) \leq d(F(p,t_1),F(q,t_1)).
\end{align*}
\end{prop}
\begin{proof}
	We show that in each step of the soul construction such a map exists. By concatenation one obtains the deformation onto the soul. Let $ X $ be noncompact (resp.\ compact with non\-empty boundary). Let $ f $ denote the Busemann function (resp.\ boundary distance function) from the proof of Lemma~\ref{lem_soulSuperlevelSets} which assumes its maximum on a set $ C_a $. By \cite[1.8]{L_OpenMap}, the gradient flow $ \Phi^t:X\to X $ associated to the concave function $ f $ is $ 1 $-Lipschitz and converges to a $ 1 $-Lipschitz function $ \Phi^\infty(p)=\lim_{t\to \infty}\Phi^t(p) $ with image in $ C_a $. This allows us to define the continuous map
	\begin{align*}
		F:X\times [0,1]\to X,\quad (p,t)\mapsto \begin{cases}
			\Phi^{\tan(t\pi/2)}(p)&\quad\text{ if } t\in [0,1),\\
			\Phi^\infty(p)&\quad \text{ if } t=1,\\
		\end{cases}
	\end{align*}
	which satisfies the desired properties of a deformation. Furthermore, it is non-expanding by the relation $ \Phi^{t_2}=\Phi^{t_2-t_1}\circ \Phi^{t_1} $ for $ 0\leq t_t\leq t_2 <1 $ and the flow being $ 1 $-Lipschitz.
\end{proof}
\begin{rem}
	In \cite[6.3]{P2} another flow is constructed (not using the modern notions of gradient curves and flows), namely a flow $ \Psi^t $ that satisfies $ f(\Psi^t(x))=f(x)+t $ (if $ f(x)+t<a $) and consists of curves $ \beta $ with $ \beta^+(t) = \nabla_{\beta(t)}f/|\nabla_{\beta(t)}f|^2 $ in contrast to the usual gradient curves $ \alpha $ with $ \alpha^+(t) = \nabla_{\alpha(t)}f $ that make up the gradient flow $ \Phi^t $. A modified retraction $ \tilde{F} $ using $ \Psi^t $ instead of $ \Phi^t $ is defined and yields the superlevel sets $ C_t = \tilde{F}(X\times\{t\}) $ from Lemma~\ref{lem_soulSuperlevelSets}.\com{How does the distance estimate $ d(\alpha_1(t),\alpha_2(t))\leq e^{\lambda t}d(\alpha_1(0),\alpha_2(0)) $ translate to the parametrizations $ \beta $? }
\end{rem}
A natural question about the soul is that of uniqueness and was studied in \cite{SConv} for Riemannian manifolds. In a compact space, the soul construction yields unique Cheeger-Gromoll souls. But in noncompact spaces, the obtained Cheeger-Gromoll soul depends on the point chosen to start the construction and may yield a different soul for each starting point. An application of \v{S}arafutdinov's retraction yields the following.
\begin{prop}\label{prop_soulsIsometric}
	In a finite-dimensional Alexandrov space of nonnegative curvature any two Cheeger-Gromoll souls are isometric.
\end{prop}
The proof of \cite{SConv} carries over to Alexandrov spaces in a straightforward way. It requires the following two lemmata of which only the latter, originally formulated for compact manifolds, needs some further arguments to generalize to Alexandrov spaces.
\begin{lem}[\text{\cite[1.6.15]{BBI}}]\label{lem_NonExpandingOnCompactIsIsometry}
	Let $ X $ be a compact metric space and $ f:X\to X $ be a surjective function. If $ f $ is $ 1 $-Lipschitz, then $ f $ is an isometry.
\end{lem}
\begin{lem}\label{lem_homotpicToIdentitySurjective}
	Let $ X $ be a compact $ n $-dimensional Alexandrov space of curvature bounded below without boundary. If $ f:X\to X $ is homotopic to the identity map, then $ f $ is surjective.
\end{lem}
\begin{proof}
	If $ n=0 $, then $ X $ is a single point and the statement is trivially true. Assume $ n\geq 1 $.
	The proof of \cite{SConv} for manifolds only relies on the two facts about the singular homology groups: $ H_n(X)\neq \{0\} $ and $ H_n(X\setminus\{x_0\})=\{0\} $ for all $ x_0\in X $ (and some fixed choice of coefficients). Note that $ X\setminus \{x_0\} $ is not compact since there are geodesics in $ X $ issuing in $ x_0 $. Hence, if $ X\setminus\{x_0\} $ is connected, we obtain the desired homology groups from Lemma~\ref{lem_homology}. Now, if $ f $ is not surjective, it can be written as a composition 
	\[ X \to X\setminus\{x_0\} \to X \]
	for some $ x_0\in X\setminus f(X) $. The induced homomorphisms 
	\[H_n(X) \to H_n(X\setminus \{x_0\}) \to H_n(X) \]
	then yield with $ H_n(X\setminus \{x_0\})=\{0\} $ that the induced homomorphism $ f_* $ by $ f $ is zero. But since the identity $ \id $ on $ X $ induces the identity isomorphism on $ H_n(X) $, the assumption $ H_n(X)\neq \{0\} $ yields $ \id_*\neq f_* $, a contradiction to the homotopy invariance of homology (cf.\ \cite{Hatcher}).
	
	It remains to show that $ X\setminus\{x_0\} $ is connected. Suppose it were not, then all points of a path-component are joined to some fixed point in a distinct path-component by a geodesic in $ X $ which passes through $ x_0 $. Since geodesics do not branch (Lemma~\ref{lem_factsAS}~\ref{lem_NonBranchingGeod}), this implies that each path-component united with $ x_0 $ is the image of a geodesic (or ray) emanating in $ x_0 $. Since $ X\setminus\{x_0\} $ is open in $ X $ and hence $ n $-dimensional (\cite[10.6.1]{BBI}), there exists an open neighborhood homeomorphic to some open subset $ \IR^n $. Therefore, $ n=1 $. But in this case, by Lemma~\ref{lem_factsAS}~\ref{lem_OneDimAS} and $ X $ being compact without boundary, $ X $ must be homeomorphic to the circle $ \IS^1 $ and removing a point still yields a connected subset. A contradiction.\com{
		The main argument could be carried out with cohomology where we directly use $ \bar{H}^n(X;\IZ/2\IZ)=\IZ/2\IZ $ (\cite[Lemma~1]{GP_RadSphere}) and that for compact Hausdorff spaces, by cite{Spanier}, its singular cohomology groups coincide with the Alexander-Spanier cohomology groups (Reference not found in cite{Spanier} but is stated in cite{Mitsu}). Maybe some standard results about cohomology yield $ H^n(X;\IZ/2\IZ)\neq \{0\} $?
	}
\end{proof}
The remaining proof works exactly as given in \cite{SConv} and is illustrated here for the convenience of the reader.
\begin{proof}[Proof of Proposition~\ref{prop_soulsIsometric}]
	Let $ S_1,S_2\subseteq X $ denote two Cheeger-Gromoll souls of a finite-dimensional Alexandrov space $ X $ of nonnegative curvature. Let $ F_1,F_2:X\times[0,1]\to X $ be the deformations obtained from Proposition~\ref{prop_sarafutdoinovsRetraction} corresponding to the soul constructions of $ S_1 $ and $ S_2 $, respectively. For $ t\in[0,1] $, we define the map
	\begin{equation*}
		h_t:S_1\to S_1,\quad p\mapsto F_1(F_2(p,t),1).
	\end{equation*}
	Then $ h_0 $ is the identity on $ S_1 $ and the maps $ (h_t)_{t\in[0,1]} $ define a homotopy to the map $ h_1 $ which is $ 1 $-Lipschitz by the non-expanding properties of $ F_1 $ and $ F_2 $ (see Figure~\ref{fig_cylSoulsIsometric} for illustration). Applying the two previous Lemmata~\ref{lem_NonExpandingOnCompactIsIsometry} and \ref{lem_homotpicToIdentitySurjective} yields that $ h_1 $ is a surjective isometry on $ S_1 $. Clearly, the maps
	\begin{align*}
		f_1:S_2\to S_1&,\quad p\mapsto F_1(p,1),\\
		f_2:S_1\to S_2&,\quad p\mapsto F_2(p,1),
	\end{align*} 
	are non-expanding. Suppose there exists points $ p,q\in S_1 $ that satisfy \linebreak $ d(f_2(p),f_2(q)) < d(p,q) $. This would imply 
	$$ d(f_1(f_2(p)),f_1(f_2(q))) \leq d(f_2(p),f_2(q)) < d(p,q), $$ which contradicts $ h_1=f_1\circ f_2 $ being an isometry. Therefore, $ f_2 $ is an isometry. Analogously, we obtain that $ f_1 $ is an isometry. 
	\begin{figure}[h]
		\centering
		\includegraphics[width=0.8\linewidth]{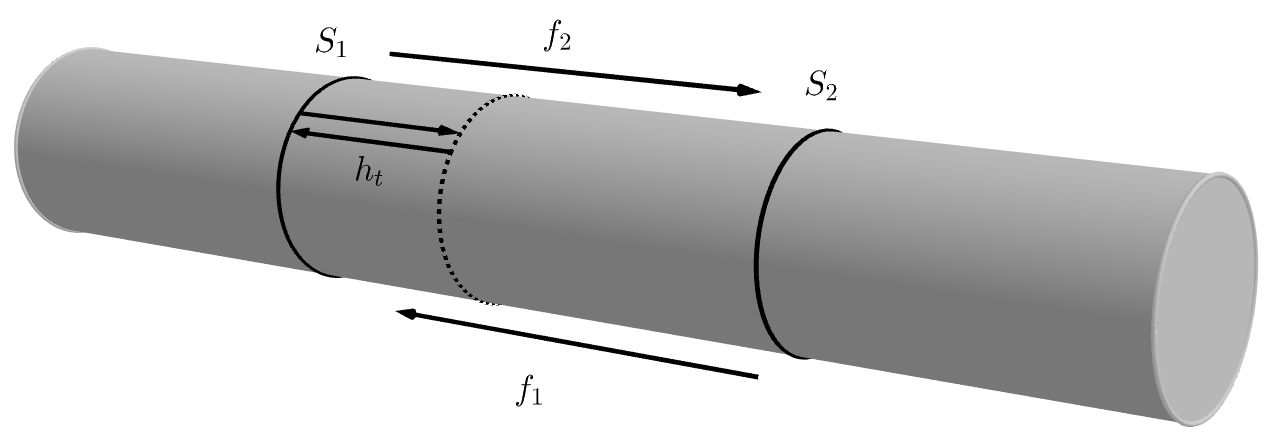}
		\caption{Two souls on a cylinder with the isometries $ f_1 $, $ f_2 $ and the homotopy $ (t,x)\mapsto h_t(x) $}\label{fig_cylSoulsIsometric}
	\end{figure}
	It remains to show that $ f_2 $ is surjective. Suppose there exists a point $ p\in S_2\setminus f_2(S_1) $. Since $ f_1 $ is injective, it would follow that $ f_1(p) $ is not in the image of $ h_1=f_1\circ f_2 $ which is a surjective map, a contradiction. Hence, $ f_1:S_1\to S_2 $ defines a surjective isometry and the souls are isometric.
\end{proof}
Note that souls that are not necessarily Cheeger-Gromoll souls can be much more abundant. For instance, in a compact convex subset of $ \IR^2 $ the Cheeger-Gromoll soul is a point and uniquely defined, but any other point is a soul, too\com{And only the points: If not a point, the soul is $ 1 $-dim and hence $ \IS^1 $ which is not possible in the plane}. Uniqueness up to isometry (or even up to homeomorphism) for these souls seems to be a more difficult question already in Riemannian manifolds (see \cite[2.1]{CG}).
\section{The Injectivity Radius in Length Spaces}\label{sec_Inj}
The usual definition of the injectivity radius of a point $ p $ in a Riemannian Manifold is the supremum of all positive radii such that the exponential map restricted to a ball of that radius is a diffeomorphism onto its image (see \cite{Lee}).
\subsection{Two Injectivity Radii}
Various generalizations of the Riemannian injectivity radius to length spaces can be found in the literature (among others \cite{Bow,Zam,Pl02,AB_GaussEqInjRadSubspaces}). Following \cite{ShSo},
we introduce two different common notions of an injectivity radius. The Riemannian injectivity radius of a point $ p $ is equal to the distance to the cut locus of $ p $, the set of points after which geodesics issuing in $ p $ stop being minimizing (see \cite[10.36]{Lee}). So, the following is a natural definition in length spaces that generalizes the Riemannian injectivity radius.
We define the distance of a point to an empty set and the infimum of an empty set to be (positive) infinity.
\begin{defn}
	Let $ X $ be a length space and $ p\in X $. We define the \textit{minimal cut locus}\index{Cut locus!minimal} $ \MinCut_X(p) $ to be the set of all points $ \gamma(t_0) $, $ t_0\in(0,1) $, that lie on a nontrivial local geodesic $ \gamma:[0,1]\to X $ such that the restriction $ \gamma|_{[0,t_0]} $ is a geodesic, but for all $ t\in(t_0,1] $ the restriction $ \gamma|_{[0,t]} $ is not. The \textit{minimal radius}\index{Minimal radius} $ \MinRad_X(p) $ of $ p $ is defined to be the distance $ d(p,\MinCut_X(p)) $ to the minimal cut locus. The \textit{minimal radius} $ \MinRad_X(A) $ of a subset $ A\subseteq X $ is the infimum $ \inf\limits_{p\in A}\MinRad_X(p) $.
\end{defn}
In Riemannian manifolds the name ``injectivity radius'' is justified by the fact that the exponential map is injective on balls of radius smaller than the injectivity radius, that is, local geodesics of length less than the injectivity radius are unique between their endpoints. In fact, this property \com{See \href{https://math.stackexchange.com/questions/3979936/two-definitions-of-the-injectivity-radius}{on MSE}}characterizes the injectivity radius: If the exponential map on a ball is injective, it is a diffeomorphism onto its image (see \cite[Theorem 2.1.12]{Kling}). Hence, the following also defines a generalization of the Riemannian injectivity radius.
\begin{defn}\label{def_inj}
	Let $ X $ be a length space and $ p\in X $. We define the \textit{cut locus}\index{Cut locus} $ \Cut_X(p) $ to be the set of all points in $ X $ which are joined to $ p $ by at least two distinct geodesics. We call those points \textit{cut points}\index{Cut point} of $ p $. The \textit{injectivity radius}\index{Injectivity radius} $ \inj_X(p) $ of $ p $ is defined to be the distance $ d(p,\Cut_X(p)) $ to the cut locus. The \textit{injectivity radius} $ \inj_X(A) $ of a subset $ A\subseteq X $ is the infimum $ \inf\limits_{p\in A}\inj_X(p) $.
\end{defn}
If the space considered is clear from context, we may omit its indication in the lower index.
\begin{rem}\label{rem_contGeodBelowInj}
	Note that the geodesics issuing in $ p $ of length less than $ \inj_X(p) $ are unique between their endpoints. In a proper metric space (e.g.\ a complete, locally compact length space) this implies that these geodesics vary continuously with their endpoints (with $ p $ fixed) \cite[ch.\ I, 3.12]{BH}.
\end{rem}
The two introduced notions can differ.
\begin{ex}
	Consider two geodesics on the sphere $ \IS^2 $ that join two antipodal points and have an angle $ <\pi $. Take the closed subset $ X $ bounded by those geodesics. This is an Alexandrov space of curvature $ 1 $ ($ \leq 1 $ and $ \geq 1 $). It satisfies $ \inj (X) = \pi $ and $ \MinRad(X) = \infty $ since all local geodesics are geodesic.
\end{ex}
Requiring extendability of geodesics and a lower curvature bound (which prohibits branching of geodesics) would give $ \inj\geq \MinRad $. But for now we can further investigate the two introduced notions without curvature bounds. Later we will focus on the setting of Alexandrov spaces of curvature bounded above. This will allow us to control conjugate points which we treat in the following section.
\subsection{Conjugate Points}
In Riemannian manifolds a central role when dealing with the injectivity radius is played by conjugate points which are defined via Jacobi Fields. An essential observation is that the exponential map fails to be a local diffeomorphism at a conjugate point and geodesics are not minimizing past conjugate points (\cite{Lee}). The latter is not necessarily the case in length spaces. Nonetheless, we will define a generalization of conjugate points and take advantage of their absence to be able to better describe the injectivity and minimal radius. 
\begin{defn}
	Let $ X $ be a metric space and $ p\in X $. We define $ \tilde{X}_p $ to be the set of all local geodesics $ c:[0,1]\to X $ with $ c(0)=p $ including the constant map $ 0_p:[0,1]\to X $ with value $ p $. We equip $ \tilde{X}_p $ with the metric
	\begin{equation*}
		\rho(c_1,c_2):=\sup_{t\in[0,1]} d(c_1(t),c_2(t)).
	\end{equation*}
	And we define the \textit{endpoint map}\index{Endpoint map}\index{Exponential map}
	\begin{equation*}
		\exp_p:\tilde{X}_p\to X,\ c\mapsto c(1). 
	\end{equation*}
\end{defn}
\begin{rem}
	The endpoint map is $ 1 $-Lipschitz, that is, $ d(\gamma(1),\sigma(1))\leq \rho(\gamma,\sigma) $ for all $ \gamma,\sigma\in \tilde{X}_p $. In particular, it is continuous.
\end{rem}
The space $ \tilde{X}_p $ serves as a generalization of the Riemannian tangent space (but usually differs from the tangent cone $ K(\Sigma_p) $ in spaces of curvature bounded below). In \cite{Ber} a Euclidean structure is introduced into this space to obtain charts via the endpoint map under assumption of bounded curvature. Without bounded curvature it may have much less structure.
\begin{defn}\label{def_conj}
	Let $ X $ be a metric space and $ \gamma:[0,1]\to X $ a local geodesic. Then $ \gamma(1) $ is said to be \textit{conjugate}\index{Conjugate point} to $ \gamma(0) $ along $ \gamma $, if there is no neighborhood $ U\subseteq \tilde{X}_p $ at $ \gamma $ such that $ \exp_p|_U $ is a homeomorphism onto its image. 
	We define the \textit{conjugate radius} $ \conj_X(p) $ of $ p $ to be the infimum over lengths of all local geodesics issuing in $ p $ that admit conjugate points. The conjugate radius $ \conj_X(A) $ of a subset $ A\subseteq X $ is the infimum $ \inf\limits_{p\in A}\conj_X(p) $.
\end{defn}
Note that $ p $ can be conjugate to itself along the constant geodesic\com{As really can happen in a space with dense points of positive infinite curvature}. In spaces of curvature bounded above this is prevented by Lemma~\ref{lem_factsAS}~\ref{lem_locMinimizing}.
This definition of conjugate points agrees with the definition in \cite{AB} but differs from the definition given in \cite{ShSo} of ``one-sided conjugate points'' (see Appendix~\ref{sec_AppKlingenberg}).
The following Lemmata illustrate a relation of conjugate points to the (minimal) cut locus.
\begin{lem}\label{lem_cutLocClosed}
	Let $ X $ be a complete, locally compact length space and $ p\in X $. Then all points in $ \overline{\Cut(p)}\setminus \Cut(p) $ are conjugate to $ p $.
\end{lem}
\begin{proof}
	Let $ q\in\overline{\Cut(p)} $ and $ (q_i)_{i\in\IN}\subseteq \Cut(p) $ be a sequence of cut points converging to $ q $. Then there exist two distinct geodesics $ \gamma_i $ and $ \sigma_i $ joining $ p $ to $ q_i $ for each $ i\in\IN $. 
	Note that these geodesics are contained in a bounded closed ball. Then Lemma~\ref{lem_factsAS}~\ref{lem_geodLimit}
	allows us to pass first to a convergent subsequence $ (\gamma_{i_k})_{k\in\IN} $ with limit $ \gamma $ and then to a convergent subsequence of $ (\sigma_{i_k})_{k\in\IN} $ with limit $ \sigma $. If $ \gamma=\sigma $, then $ q $ is a conjugate point along $ \gamma $, since in each neighborhood of $ \gamma $ we can find two distinct geodesics $ \gamma_i,\sigma_i $ with the same endpoint $ q_i $. If $ \gamma\neq\sigma $, then $ q $ is a cut point of $ p $.
\end{proof}
The following lemma was similarly formulated in \cite[7.15]{ShSo} and \cite[Theorem~I]{ZamCutLoci}.
\begin{lem}\label{lem_MinCutminCutConj}
	Let $ X $ be a complete, locally compact length space and $ p\in X $. Then all points in $ \MinCut(p)\setminus \Cut(p) $ are conjugate to $ p $.\com{This holds for $ \overline{\MinCut}\setminus \Cut $  too}
\end{lem}
\begin{proof}
	Let $ q\in\MinCut(p) $, then there exists a local geodesic $ \gamma $ from $ p $ through $ q $ that is a geodesic between $ p $ and $ q $ but not past $ q $. Let $ (q_i)_{i\in\IN} $ be a sequence of points on $ \gamma $ past $ q $ with limit $ q $ and consider geodesics $ \sigma_i $ joining $ p $ to $ q_i $. 
	There exists a convergent subsequence with limit geodesic $ \sigma $ joining $ p $ to $ q $. If the geodesic $ \sigma $ is distinct from $ \gamma $, then $ q\in \Cut(p) $. Otherwise, $ q $ is conjugate to $ p $ along $ \gamma $.
\end{proof}
We conclude this section with a fundamental observation about the injectivity radius in the absence of conjugate points. In Riemannian manifolds the injectivity radius can be controlled by conjugate points and closed local geodesics \cite{Kling1959}. This was generalized to length spaces in varying formulations (e.g.\ in \cite[2.13]{Bow},\cite[8.1]{ShSo}). Note that there are problems with the proof of \cite[8.1]{ShSo} whose correction we address in Appendix~\ref{sec_AppKlingenberg}. We give a proof that makes direct use of the definition of conjugate points.
\begin{lem}[Klingenberg Lemma]\label{lem_Klingenberg}
	Let $ X $  be a complete, locally compact length space and $ p\in X $. If $ q\in \overline{\Cut(p)} $ is a point with $ \inj(p)=d(p,q)>0 $, then $ q $ is conjugate to $ p $ or it is the midpoint of some local geodesic from $ p $ to $ p $ which has length $ 2\inj(p) $.
\end{lem}\com{If we replace ``$ q $ conjugate to $ p $'' by ``some point $ r $ is conjugate to $ p $ with $ d(r,p)\leq \MinRad(p) $'' this holds for $ \MinRad $ instead of $ \inj $.}
\begin{proof}
	Assume that $ q $ is not conjugate to $ p $. Then by Lemma~\ref{lem_cutLocClosed}, $ q $ is a cut point of $ p $ and there exist two distinct geodesics $ \gamma, \sigma: [0,1]\to X$ joining $ p $ to $ q $. We want to show that the concatenation of $ \gamma $ and the reversal $ -\sigma $ is a local geodesic (of length $ 2 d(p,q) $), that is, there exists $ \varepsilon>0 $ with $ d(\gamma(1-\varepsilon),\sigma(1-\varepsilon))=2\varepsilon d(p,q) $.
	Assume, on the contrary, that there exists a null sequence $ (\varepsilon_i)_{i\in\IN} $ such that, for all $ i\in \IN $, we have
	\begin{equation}\label{eq_klingNotAnglePi}
		d\left(\gamma(1-\varepsilon_i),\sigma(1-\varepsilon_i)\right)<2\varepsilon_i d(p,q).
	\end{equation}
	Let $ c_i:[0,1]\to X $ be a geodesic that joins $ \gamma(1-\varepsilon_i) $ to $ \sigma(1-\varepsilon_i) $. Since $ q $ is not conjugate to $ p $, there exist neighborhoods $ U_\gamma, U_\sigma\subseteq \tilde{X}_p $ at $ \gamma,\sigma $ such that the restriction of the endpoint map on those neighborhoods is a homeomorphism onto its image. We can choose $ U_\gamma $ and $ U_\sigma $ to be disjoint and $ i\in\IN $ sufficiently large such that the image of $ c_i $ is contained in $\exp_p(U_\gamma)$ and $ \exp_p(U_\sigma) $. We choose a radius $$ r:=(1-\eps_i)d(p,q)+\frac12 \cdot d\left(\gamma(1-\varepsilon_i),\sigma(1-\varepsilon_i)\right). $$ By the triangle inequality, $ c_i $ has image in the closed ball of radius $ r $ at $ p $. We conclude with \eqref{eq_klingNotAnglePi} that $ r<d(p,q)=\inj(p) $ and hence in this ball geodesics issuing in $ p $ are unique and vary continuously with their endpoints (see Remark~\ref{rem_contGeodBelowInj}). Therefore, the geodesics $ \alpha_t:[0,1]\to X$ that join $ p $ to $ c_i(t) $ are unique and $ t\mapsto \alpha_t $ defines a continuous curve in $ \tilde{X}_p $ which joins (the linear reparametrizations of) $ \gamma|_{[0,1-\eps_i]} $ in $ U_\gamma $ to $ \sigma|_{[0,1-\eps_i]} $ in $ U_\sigma $. With the homeomorphism $ \exp_p|_{U_\gamma} $ the curves $ \alpha_t $ must agree with the curves $ (\exp_p|_{U_\gamma})^{-1}(c_i(t)) $ as long as $ \alpha_t\in U_\gamma $ (see Figure~\ref{fig_endpointMapKlingenberg} for illustration). But this implies that the curve $ t\mapsto \alpha_t $ cannot leave $ U_\gamma $, a contradiction.
	The main statement follows. 
	\begin{figure}[h]
		\centering
		\begin{minipage}[t]{0.38\textwidth}
			\vspace*{\fill}
			\centering
			\vfill
			\includegraphics[width=0.75\linewidth]{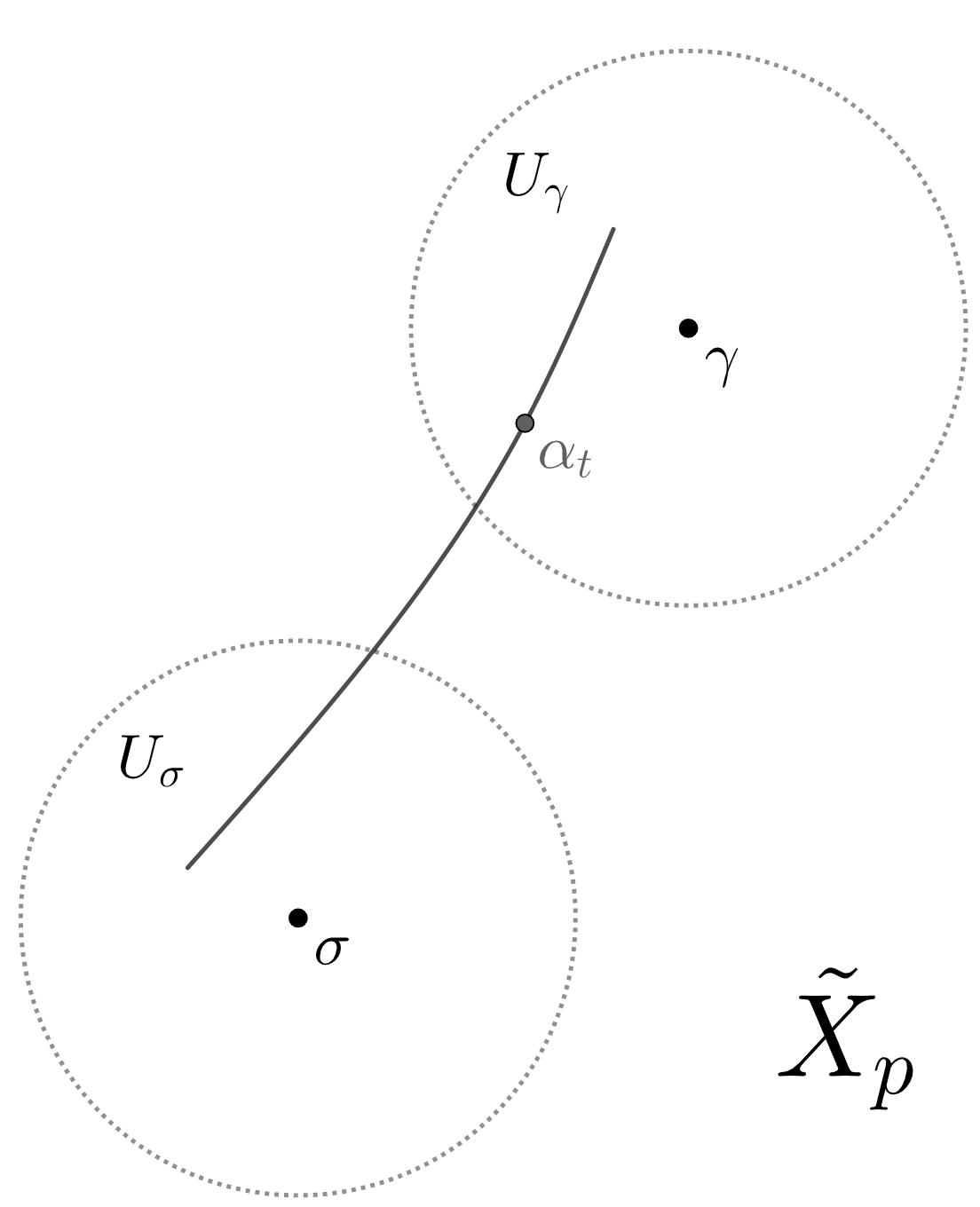}
			\vfill
		\end{minipage}
		\qquad
		\begin{minipage}[t]{0.1\textwidth}
			\vspace*{2cm}
			\centering
			\huge{$ \overset{\exp_p}{\longrightarrow} $}
		\end{minipage}
		\qquad
		\begin{minipage}[t]{0.38\textwidth}
			\vspace*{\fill}
			\centering
			\vfill
			\includegraphics[width=0.9\linewidth]{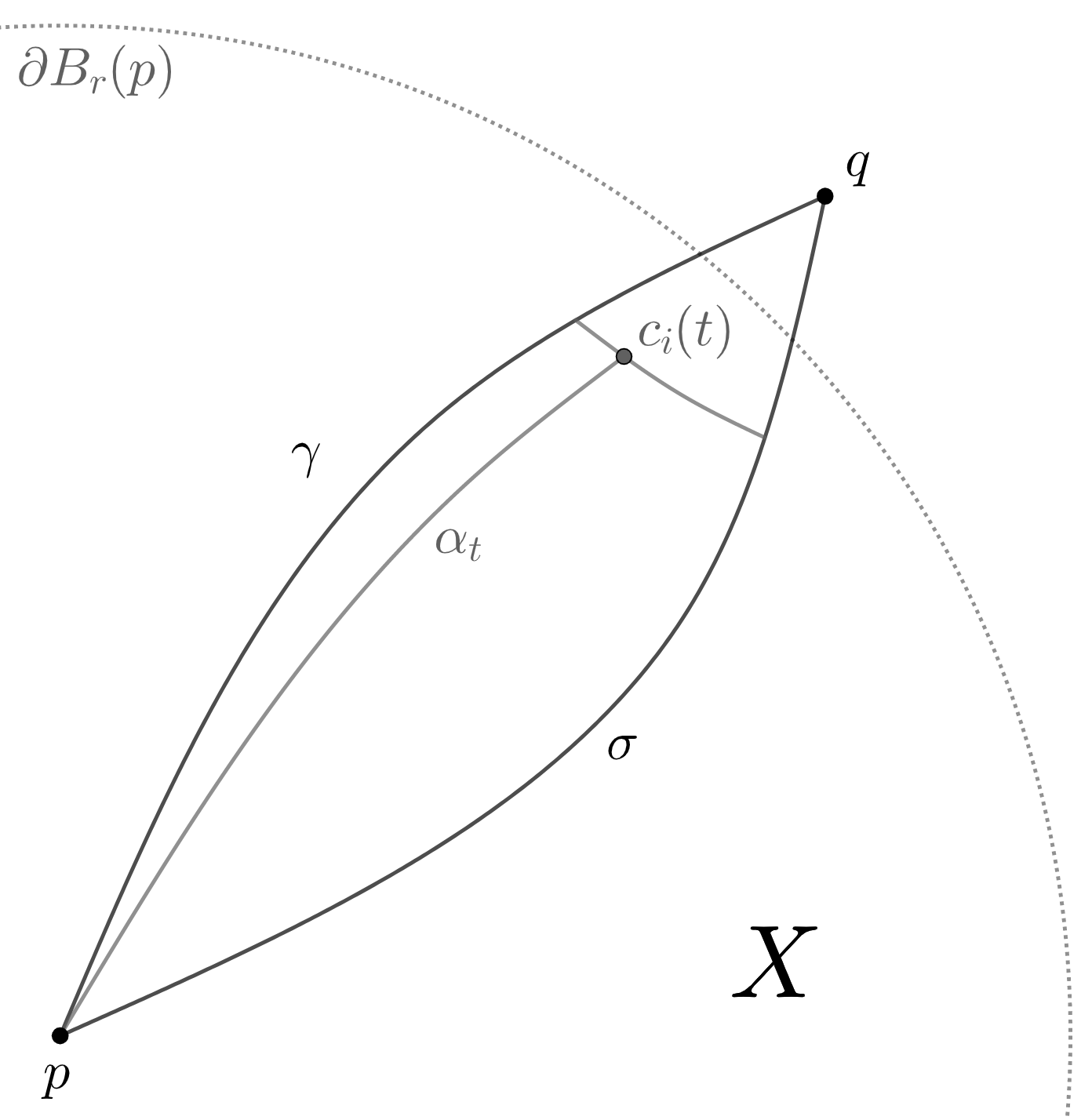}
			\vfill
		\end{minipage}
		\caption{The setup in the proof of Lemma~\ref{lem_Klingenberg} leading to a contradiction}\label{fig_endpointMapKlingenberg}
	\end{figure}
	To prove the last claim consider $ J:=\{t\in[0,1]: \alpha_t = (\exp_p|_{U_\gamma})^{-1}(c_i(t)) \} $. Note that $ 0\in J $ and by continuity of the involved maps, $ J $ is closed in $ [0,1] $. And since $ U_\gamma $ is open and $ \exp_p|_{U_\gamma} $ is injective, $ J $ is also open. Consequently, $ J=[0,1] $ and hence $ \alpha_t\in U_\gamma $ for all $ t\in[0,1] $.
\end{proof}
If the injectivity radius or the minimal radius is determined by a local geodesic loop, we have a much better control over it. This will be crucial to Theorem~\ref{thm_injSoul}.

\subsection{The Injectivity Radius Before Conjugate Points}

We can now establish some properties of the injectivity radius in the absence of conjugate points.
The following two statements were similarly stated in \cite[2.16+2.18]{Bow} and \cite[7.15]{ShSo}.
\begin{cor}\label{cor_ineqInjMinRad}
	Let $ X $ be a complete, locally compact length space and $ p\in X $. Then the inequality $$ \inj(p) \leq \MinRad(p) $$ holds whenever one of those values is less than $ \conj(X) $.
\end{cor}
\begin{proof}
	We can assume $ \MinRad(p)<\conj(X) $. Otherwise, the stated inequality already holds true or both values are greater than $ \conj(X) $. By Lemma~\ref{lem_MinCutminCutConj} we have $ \MinCut(p)\cap B_{\conj(X)}(p)\subseteq \Cut(p) $ and the statement follows.
\end{proof}
\com{It may be possible to generalize the inequality to hold independently of $ \conj(X) $. One would have to answer the following question. Are conjugate points along a geodesic in the boundary of the cut locus? Or: Give two distinct local geodesics joining the same points an upper bound for $ \inj $ as in the following but independently of $ \conj(X) $? If not, can we find an example of $ \inj(p) > \MinRad(p) $}
\begin{cor}\label{cor_injBoundLocGeod}
	Let $ X $ be a complete, locally compact length space and $ p\in X $. If there exist $ q\in X $ and two distinct local geodesics $ \gamma $ and $\sigma $ joining $ p $ to $ q $, then the inequality \[\inj(p)\leq \frac{L(\gamma)+L(\sigma)}{2} \]
	holds whenever $ \inj(p) $ or $ \frac{L(\gamma)+L(\sigma)}{2} $ is less than $ \conj(X) $.
\end{cor}
\begin{proof}
	In case $ \gamma $ and $ \sigma $ are geodesic we have $ \inj(p)\leq L(\gamma)= (L(\gamma)+L(\sigma))/2$. Now assume one of them, say $ \gamma $, is not geodesic and without loss of generality longer than $ \sigma $. It cannot be geodesic past the length $ (L(\gamma)+L(\sigma))/2 $, since the concatenation of $ \sigma $ and the end part of $ \gamma $ is shorter. Corollary~\ref{cor_ineqInjMinRad} implies that $ \inj(p)\leq \MinRad(p) \leq (L(\gamma)+L(\sigma))/2 $.
\end{proof}

To conclude the section we list two consequences of Lemma~\ref{lem_Klingenberg} in the absence of conjugate points.
\begin{cor}\label{cor_injSubsets}
	Let $ X $ be a complete, locally compact length space, $ C\subseteq X $ be a closed totally convex subset and $ p\in C $. If $ \inj_X(p)\in(0,\conj(X)) $, then $$ \inj_X (p)= \inj_C(p). $$
\end{cor}
\begin{proof}
	By definition we have $ \inj_X(p)\leq\inj_C(p) $. Since $ X $ is proper, there exists a $ q\in \overline{\Cut_X(p)} $ with $ \inj_X(p)=d(p,q) $ and by assumption $ \inj_X(p)>0 $. This allows us to apply Lemma~\ref{lem_Klingenberg}.
	Therefore, $ \inj_X(p) $ is realized by half the length of a local geodesic $ \gamma $ issuing and terminating in $ p $. By total convexity, $ \gamma $ is contained in $ C $ and hence 
	\begin{equation*}
		\inj_C(p)\leq \frac{1}{2} L(\gamma) = \inj_X(p). \qedhere
	\end{equation*}
\end{proof}

\begin{cor}\label{cor_injMinRad}
	Let $ X $  be a complete, locally compact length space and $ p\in X $. If $ \inj(p)\in(0,\conj(X)) $, then $$ \inj(p)=\MinRad(p). $$\com{Without the assumption we could have $ \inj(p)=0  $ but $ \MinRad(p)>0 $ positive. Example?} 
\end{cor}
\begin{proof}
	Corollary~\ref{cor_ineqInjMinRad} gives $ \inj(p)\leq \MinRad(p) $. Analogously to the proof of the previous Corollary~\ref{cor_injSubsets} we can find a point $ q\neq p $ that realizes the injectivity radius of $ p $. Then by Lemma~\ref{lem_Klingenberg}, the injectivity radius is given by half the length of a local geodesic loop which also gives an upper bound for the minimal radius and hence $ \MinRad(p)\leq \inj(p) $. 
\end{proof}

\subsection{Conjugate Points and an Upper Curvature Bound}
An upper curvature bound allows for a better treatment of conjugate points and the injectivity radius.
For one, it ensures nonzero injectivity radius for a given point. \com{$ \inj $ and $ \MinRad $ are bounded below by the smallest CAT($ K $) neighborhood.}\com{Example for $ \inj(p)=0 $ is the doubling of the disk where $ p $ is a boundary point}
\begin{lem}\label{lem_injNonzero}
	Let $ X $ be a metric space of curvature $ \leq K $ and $ p\in X $, then $ \inj(p)>0 $ and $ \MinRad(p)>0 $.
\end{lem}
\begin{proof}
	At each point of a space of curvature $ \leq K $ there exists a neighborhood which is CAT($ K $). The statement then follows from Lemma~\ref{lem_factsAS}~\ref{lem_locMinimizing}.
\end{proof}
In the previous section the advantages of avoiding conjugate points became apparent. The following shows that with an upper curvature bound they only appear after a certain distance.
This is essentially due to the following. A curve in an Alexandrov space of curvature $ \leq K $ admits a neighborhood which is locally isometric to an open subset of a CAT($ K $) space. This space is constructed by gluing a chain of small CAT($ K $) neighborhoods consecutively together. According to the authors of \cite{AKP_Found}, this is due to Alexander Lytchak.
\begin{prop}[\text{\cite[9.47]{AKP_Found}} Patchwork Along a Curve]\label{prop_Patchwork}
	Let $ X $ be an Alexandrov space of curvature $ \leq K $ and $ c:[a,b]\to X $ a curve. Then there is an Alexandrov space $ \mathcal{N} $ that is CAT($ K $), an open subset $ \hat{\Omega}\subseteq \mathcal{N} $, a curve $ \hat{c}:[a,b]\to \hat{\Omega} $ and an open locally isometric map $ \Phi:\hat{\Omega}\to X $ such that $ \Phi\circ \hat{c} = c $.
\end{prop}
Recall that, local geodesics of length $ <D_K $ in a CAT($ K $) space are unique between their endpoints and vary continuously with their endpoints. We note consequences of the patchwork along a curve.
By a continuous family of curves $ (\gamma_t)_{t\in M} $ for some metric space $ M $ we denote a continuous map $ t\mapsto \gamma_t $ w.r.t.\ the supremum metric on the space of curves.
\begin{prop}[\text{\cite[9.46]{AKP_Found}}]\label{prop_symmConjPointsAfterPi}
	Let $ X $ be an Alexandrov space of curvature $ \leq K $ and $ \gamma:[0,1]\to X $ a local geodesic of length $ <D_K $. Then there exist neighborhoods $ U_0 $ at $ \gamma(0) $ and $ U_1 $ at $ \gamma(1) $ and a unique continuous family of local geodesics $ (\gamma_{p,q})_{(p,q)\in U_0\times U_1} $ with $ \gamma_{\gamma(0),\gamma(1)}=\gamma $.
\end{prop}
\begin{cor}[\text{\cite[Theorem~3]{AB}}]\label{cor_conjPointsAfterPi}
	Let $ X $ be an Alexandrov space of curvature $ \leq K $. Then $ \conj(X)\geq D_K. $
\end{cor}
\begin{cor}\label{cor_firstVarLocGeod}\index{First variation!on local geodesics}
	Let $ X $ be a locally compact Alexandrov space of curvature $ \leq K $, $ \gamma :[0,1]\to X $ be a local geodesic of length less than $ D_K $ and $ c:[0,1]\to X $ be a geodesic with $ c(0)=\gamma(1) $. Then there exists $ \delta>0 $ and a unique continuous family of local geodesics $ (\gamma_s:[0,1]\to X)_{0\leq s < \delta} $ joining $ \gamma(0) $ to $ c(s) $ with $ \gamma_0=\gamma $. Furthermore, the first variation formula holds, i.e.\
	\begin{equation*}
		\lim_{s\searrow 0} \frac{L(\gamma_s)-L(\gamma)}{s} = -\cos \angle_{\gamma(1)}(-\gamma',c'),
	\end{equation*}
	where $ -\gamma',c'\in\Sigma_{\gamma(1)} $ denote the directions corresponding to the reversal $ -\gamma $ and $ c $, respectively.
\end{cor}
\begin{proof}
	We apply patchwork along the curve $ \gamma $ and the standard first variation formula \cite[9.36]{AKP_Found} for CAT($ K $) spaces. By the local isometry the angles and lengths of curves remain unchanged.
\end{proof}
\com{The first variation formula still holds true for any (Lipschitz-)continuous family of local geodesics (see my notes). Maybe it is even true for any continuous family (see \cite[9.38]{AKP_Found}), not sure what the proof is though.}
To investigate continuity of the injectivity radius we define the interval $ [0,\infty] $ to be the one-point compactification of $ [0,\infty) $, that is, maps $ f :X\to [0,\infty] $ are continuous at a point $ p\in X $ with $ f(p)=\infty $ if and only if $ \lim\limits_{p'\to p}f(p')=\infty. $
\begin{prop}\label{prop_InjCont}
	Let $ X $ be a locally compact Alexandrov space of curvature $ \leq K $ and $ p\in X $ with $ \inj(p)<D_K $. Then the injectivity radius $ \inj:X\to[0,\infty] $ is continuous at $ p $.
\end{prop}
\begin{proof}
	Let $ p\in X $ with $ \inj(p) < D_K $ and $ (p_i)_{i\in\IN} $ be a sequence converging to $ p $. We first show $ \limsup_{i\to\infty} \inj(p_i) \leq \inj(p) $.\\
	Let $ \varepsilon>0 $ such that $ \inj(p)+\varepsilon < D_K $. Then there exists a cut point $ q\in X $ of $ p $ with $ d(p,q)<\inj(p)+\varepsilon $. Let $ \gamma,\sigma\in\tilde{X}_q $ be two distinct geodesics joining $ q $ to $ p $. By $ d(p,q)<D_K $ and Corollary~\ref{cor_conjPointsAfterPi}, the point $ p $ cannot be conjugate to $ q $ along $ \gamma $ or $ \sigma $ and hence there are neighborhoods $ U_\gamma \subseteq \tilde{X}_q $ at $ \gamma $ and $ U_\sigma \subseteq \tilde{X}_q $ at $ \sigma $ which are homeomorphic to a neighborhood of $ q $ in $ X $ by the endpoint map $ \exp_p $. Since $ \gamma $ and $ \sigma $ are distinct, we can choose $ U_\gamma $ and $ U_\sigma $ to be disjoint. We consider, for sufficiently large $ i\in\IN $, those local geodesics $ \gamma_i\in U_\gamma,\sigma_i\in U_\sigma $ that join $ q $ to $ p_i $. They converge to $ \gamma $ and $ \sigma $, respectively. And by Lemma~\ref{lem_factsAS}~\ref{lem_lengthContLocGeod}, the lengths $ L(\gamma_i) $ and $ L(\sigma_i) $ converge to $ L(\gamma) $ and $ L(\sigma) $, respectively. With Corollary~\ref{cor_injBoundLocGeod}, we conclude
	$$ \limsup_{i\to\infty}\inj(p_i)\leq \limsup_{i\to\infty}\frac{L(\gamma_i)+L(\sigma_i)}{2} = d(p,q) < \inj(p) +\varepsilon. $$ 
	Since this holds for all sufficiently small $ \varepsilon>0 $, we obtain $$ \limsup\limits_{i\to\infty} \inj(p_i) \leq \inj(p). $$
	We show $ \liminf_{i\to\infty} \inj(p_i) \geq \inj(p) $. Assume that, on the contrary, there exists a subsequence $ (p_{i_n})_{n\in\IN} $ with $ \inj(p_{i_n})<\inj(p)-\varepsilon $. Then there exist points $ q_n $ such that $ d(p_{i_n},q_n)<\inj(p)-\varepsilon $ and $ p_{i_n} $ and $ q_n $ are joined by two distinct geodesics $ \gamma_n $ and $ \sigma_n $. Note that for sufficiently large $ k\in\IN $, all considered geodesics are contained in a bounded set. Then Lemma~\ref{lem_factsAS}~\ref{lem_geodLimit} allows us to pass to subsequences with limit geodesics $ \gamma $ and $ \sigma $, respectively, with common endpoints $ p $ and some $ q\in X $. We cannot have $ p=q $ since $ p $ has a uniquely geodesic neighborhood by Lemma~\ref{lem_factsAS}~\ref{lem_locMinimizing}. Furthermore, we have
	$$ L(\gamma)= d(p,q) = \lim_{n\to\infty} d(p_{i_n},q_n) \leq \inj(p) - \eps <D_K. $$
	If $ \gamma\neq \sigma $, then $ q $ is a cut point of $ p $ at distance less than $ \inj(p) $, a contradiction. If $ \gamma=\sigma $, we have found two distinct families of geodesics $ (\gamma_{m_n})_{n\in\IN} $ and $ (\sigma_{m_n})_{n\in\IN} $ that are continuous at $ \gamma=\sigma $ which has length $ <D_K $. This contradicts Proposition~\ref{prop_symmConjPointsAfterPi}. Therefore, the assumption was false and we obtain 
	$$ \liminf_{i\to\infty}\inj(p_i)\geq \inj(p). $$ 
	We showed semi upper and lower continuity in $ p $ and can conclude continuity in $ p $.
\end{proof}
We define $ [0,\infty] $ to be the one-point compactification of $ [0,\infty) $, that is, functions $ f:\IR\to [0,\infty] $ are continuous at a point $ p\in \IR $ with $ f(p)=\infty $ if and only if for each sequence $ (p_i)_{i\in\IN} $ with limit $ p $ the sequence $ (f(p_i))_{i\in\IN} $ diverges to $ +\infty $.
\begin{cor}\label{cor_injCont}
	Let $ X $ be a locally compact Alexandrov space of curvature $ \leq K $. Then the map $ f:X\mapsto [0,D_K],\ p\mapsto \min\{D_K,\inj(p)\} $ is continuous.
\end{cor}
\begin{proof}
	Let $ p\in X $. If $ f(p)<D_K $, we have continuity by Proposition~\ref{prop_InjCont}. We assume from now on $ f(p)=D_K $. Assume $ D_K<\infty $ and let $(p_i)_{i\in\IN}\subseteq X $ be a sequence converging to $ p $. If $ f $ is not continuous at $ p $, there exists $ \varepsilon>0 $ and a subsequence $ (p_{i_n})_{n\in\IN} $ such that $ f(p_{i_n})=\inj(p_{i_n}) < D_K-\varepsilon $ for all $ n\in \IN $. We repeat the second part of the proof of Proposition~\ref{prop_InjCont} to obtain $ \inj(p) \leq \liminf_{n\to\infty}\inj(p_{i_n}) < D_K $. A contradiction to $ f(p)=D_K $. In case $ D_K=\infty $ the statement follows from an analogous argument where one replaces the upper bounds $ D_K-\varepsilon $ by a finite number $ C\in\IR $.
\end{proof}

We conclude this section with the observation that Theorem~\ref{thm_injSoul} applies to the minimal radius in place of the injectivity radius as well.
\begin{cor}\label{cor_ThmMinRad}
	Let $ K> 0 $ and $ X $ be a finite-dimensional Alexandrov space of non\-negative curvature bounded above by $ K $ and $ S\subseteq X$ be a Cheeger-Gromoll soul. If $ \MinRad(X)<\frac{\pi}{\sqrt{K}} $, then 
	\begin{equation*}
		\MinRad(X) = \MinRad(S).
	\end{equation*}
\end{cor}
\begin{proof}
	Using Corollary~\ref{cor_injMinRad}, the assumption $ \MinRad(X)<\pi/\sqrt{K} $ yields $\inj(X)<\pi/\sqrt{K} $ and we can apply Theorem~\ref{thm_injSoul} to obtain $$ \inj(S) = \inj(X) < \frac{\pi}{\sqrt{K}}. $$ Applying Lemma~\ref{lem_injNonzero} and Corollary~\ref{cor_conjPointsAfterPi} to $ X $ and $ S $ each, we obtain $ \inj(X)=\MinRad(X) $ and $ \inj(S)=\MinRad(S) $. We conclude 
	\begin{equation*}
		\MinRad(X) = \inj(X) = \inj(S) = \MinRad(S). \qedhere
	\end{equation*} 
\end{proof}


\section{The Injectivity Radius of a Soul}\label{sec_InjSoul}
We turn to proving Theorem~\ref{thm_injSoul}. Corollary~\ref{cor_injSoulPos} then follows with Proposition~\ref{prop_soulPoint}.
We follow the proof of \cite{CK} and will see that the Riemannian techniques more or less can be replaced by their counterparts in Alexandrov spaces.
To begin the proof, let $ K> 0 $ \com{We can do the proof with $ K\geq 0 $ as well but we gain nothing and the proof for continuity of the function $ g $ complicates when taking care of the case $ D_K=\infty $.}and $ X $ be a finite-dimensional Alexandrov space of non\-negative curvature bounded above by $ K $. Following the soul construction of Proposition~\ref{prop_soul}, we obtain the iterated spaces $ (X_l)_{l=0,\dots,n} $ where $ X_0:=X $ and $ X_n $ is the soul. Fix $ l\in\{0,\dots,n-1\} $ and consider the sets $ (C_t)_{t\in I} $ between $ X_{l+1}=C_{a} $ and $ X_{l} $ obtained in Lemma~\ref{lem_soulSuperlevelSets}, where  $ I=[a,\infty) $  or $ I=[a,0] $. We define the function
\begin{equation}\label{eq_g}
	g:I\to[0,\infty],\quad t\mapsto \inf\{\inj_{X_l}(p):p\in C_t\}.
\end{equation} 
In the following we omit the index $ X_l $. The next Lemma~covers the main part of the proof.
\begin{lem}\label{lem_gLeftDerivativeZero}
	Let $ g:I\to [0,\infty] $ as defined in (\ref{eq_g}) and $ t\in I\setminus\{a\} $. If $ g(t)<D_K $, then the left derivative satisfies $ g^-(t)=0. $
\end{lem}
\begin{proof}
	Since $ p\mapsto \min\{D_K,\inj(p)\} $ is continuous (Corollary~\ref{cor_injCont}) and $ C_t $ is compact, there exists a point $ q\in C_t $ with $ \inj(q)=g(t) $. If $ g $ is constant in a neighborhood around $ t $, then already $ g^-(t)=0 $. Since $ g $ is non\-increasing, we can assume that $ g(s)>g(t) $ for all $ s\in [a,t) $. Therefore, the point $ q $ cannot lie in $ C_s $ for $ s\in[a,t) $ and hence $ q\in \partial C_t $. 
	Since X is proper, there exists a point $ \overline{q}\in\overline{\Cut(q)} $ with $ \inj(q)=d(q,\overline{q}) $ which is distinct to $ q $ by Lemma~\ref{lem_injNonzero}. Then Lemma~\ref{lem_Klingenberg} yields a local geodesic $ \gamma $ of length $ 2\inj(q) $ issuing and terminating in $ q $ with midpoint $ \overline{q} $. Now, choose a geodesic $ c $ from $ q $ to some point in $ C_{a} $ and let $ \gamma^1 $ and $ \gamma^2 $ denote the two distinct geodesics from $ q $ to $ \overline{q} $ that are part of $ \gamma $. 
	\begin{figure}[h]
		\centering
		\includegraphics[width=0.95\linewidth]{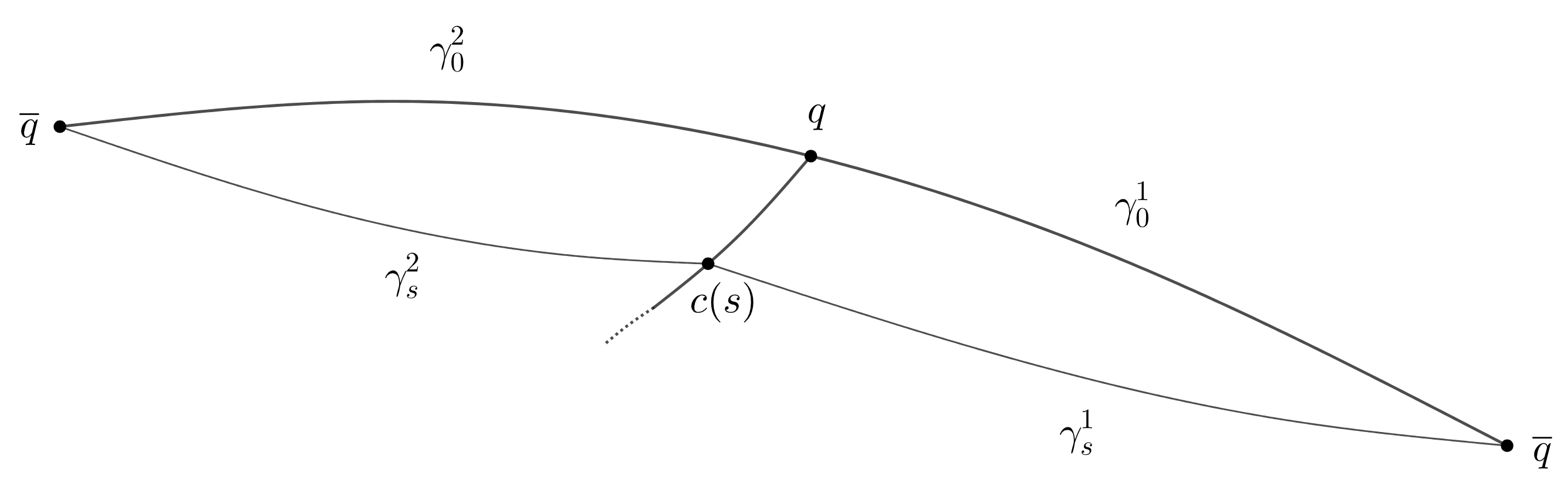}
		\caption{A sketch of $ \gamma $ and the local geodesics $ \gamma_s $ from $ \overline{q} $ to $ c(s) $ if one identifies the two copies of $ \overline{q} $}\label{fig_CKFirstVariation}
	\end{figure}
	They both have length $ L(\gamma^i) = \inj(q) < D_K $. Hence, we can apply Corollary~\ref{cor_firstVarLocGeod} to obtain $ \delta>0 $ and continuous families of local geodesics $ (\gamma^1_s:[0,1]\to X)_{s\in[0,\delta)} $ with $ \gamma^1_0=\gamma^1 $ and $ (\gamma^2_s:[0,1]\to X)_{s\in[0,\delta)} $ with $ \gamma^2_0=\gamma^2 $ (see Figure~\ref{fig_CKFirstVariation}). 
	By continuity of the length of local geodesics (Lemma~\ref{lem_factsAS}~\ref{lem_lengthContLocGeod}), we can choose $ \delta>0 $ sufficiently small such that $ L(\gamma^1_s) $ and $ L(\gamma^2_s) $ are less than $ D_K $ for all $ s\in[0,\delta) $.
	We define the function 
	\begin{equation*}
		f:[0,\delta)\to [0,D_K),\quad s\mapsto \frac{L(\gamma^1_s)+L(\gamma^2_s)}{2}
	\end{equation*}
	which satisfies $ f(0)=\inj(q)=g(t) $ and $ f(s)\geq \inj(c(s)) $ (see Corollary~\ref{cor_injBoundLocGeod}). Let $ \alpha $ be the angle between $ \gamma^1 $ and $ c $ and $ \beta $ be the angle between $ \gamma^2 $ and $ c $. Since $ \gamma^1 $ and $ \gamma^2 $ make up the geodesic $ \gamma $ in $ q $, we obtain from the Adjacent Angle Lemma~\ref{lem_factsAS}~\ref{lem_adjAngle} $ \alpha+\beta = \pi $. Another consequence from Corollary~\ref{cor_firstVarLocGeod} is that the derivative of $ f $ in $ 0 $ is given by the first variation formula. 
	\begin{align}\label{eq_fVariation}
		\lim_{s\searrow0}\frac{f(s)-f(0)}{s} = \frac{1}{2}\left(-\cos(\alpha) - \cos(\beta)\right) = 0.
	\end{align}
	To transfer the first order behavior of $ f $ to $ g $, we need to know the relation between the points $ c(s) $ and the sets $ C_t $. To establish that relation, define $ r(s):=d(c(s),\partial C_t) $. By Lemma~\ref{lem_soulSuperlevelSets}, we have $ c(s)\in C_{t-r(s)} $. 
	Since this geodesic ends in $ C_{a} $, there exists $ s_0>0 $ with $ r(s_0)=t-a $. By definition, we also have $ r(0)=0 $. The function $ r $ is, by Proposition~\ref{prop_distConcave}, concave and thus lies not below the linear function $ s\mapsto \frac{t-a}{s_0} s $ for $ s\in [0,s_0] $. Setting $ H:=\frac{t-a}{s_0} $, we obtain $ r(s)\geq Hs $ for $ s\in[0,s_0] $ and hence $ c(s)\in C_{t-r(s)} \subseteq C_{t-Hs} $.
	We observe $$ f(s)\geq \inj(c(s))\geq \inj(C_{t-Hs}) = g(t-Hs) $$ and conclude with (\ref{eq_fVariation}) that
	\begin{align*}
		\limsup_{h\searrow 0} \frac{g(t)-g(t-h)}{h}&=\limsup_{h\searrow 0} \frac{f(0)-g(t-h)}{h}\\
		&= \limsup_{h\searrow 0} \frac{f(0)-g(t-Hh)}{Hh}\\
		&\geq \limsup_{h\searrow 0} \frac{f(0)-f(h)}{Hh}\\
		&= -\frac{1}{H}\limsup_{h\searrow 0} \frac{f(h)-f(0)}{h}= 0.
	\end{align*}
	Since $ g $ is non\-increasing, we have $ g^-(t)=0 $.
\end{proof}
The following Lemma~will give a continuity property that, together with the preceding Lemma, will ensure that $ g $ is constant whenever it is below $ D_K $.
\begin{lem}\label{lem_gClosed}
	Let $ g:I\to [0,\infty] $ as defined in (\ref{eq_g}). Then the map $$  \bar{g}:I\to [0,D_K],\quad t\mapsto \min \{D_K, g(t)\}  $$ is continuous.
\end{lem}
\begin{proof}
	Recall that Corollary~\ref{cor_injCont} showed that
	\begin{equation*}
		f:X_l\to [0,D_K],\ p\mapsto \min\{D_K,\inj_{X_l}(p)\}
	\end{equation*}
	is continuous. Note that here $ D_K<\infty $ since $ K>0 $. For a metric space $ Y $, we denote by $ \mathcal{K}(Y) $ the set of all non\-empty compact subsets of $ Y $ endowed with the Hausdorff metric. We split $ \bar{g} $ into the three following functions,
	\begin{align*}
		&g_1:I\to \mathcal{K}(X_l),\quad t \mapsto C_t, \\
		&g_2:\mathcal{K}(X_l) \to \mathcal{K}(\IR), \quad A \mapsto f(A), \\
		&g_3:\mathcal{K}(\IR) \to \IR, \quad B \mapsto \min B,
	\end{align*}
	and obtain $ \bar{g}=g_3\circ g_2\circ g_1 $. We changed the order in which the minimum is taken but the equality still holds. The first map, $ g_1 $, is continuous by Lemma~\ref{lem_soulSuperlevelSets}~\ref{lem_CtCont}. The second and the third map, $ g_2 $ and $ g_3 $, are continuous by arguments about the Hausdorff distance as shown in Lemma~\ref{lem_HausdorffImagemapCont} and Lemma~\ref{lem_HausdorffMinContinuous}. Hence, their concatenation $ \bar{g} $ is continuous.
\end{proof}

\begin{proof}[Proof of Theorem~\ref{thm_injSoul}]
	We construct the soul according to the proof of Proposition~\ref{prop_soul}. We obtain finitely many $ (X_l)_{l=0,\dots,n} $ with $ X_0=X $ and the soul $ S=X_n $. We prove by induction that $ \inj_{X_n}(X_n)=\inj(X) $. 
	Fix $ l\in\{0,\dots,n-1\} $ and assume $ \inj_{X_l}(X_l)=\inj_X(X) $. We show $ \inj_{X_l}(X_{l+1})=\inj_{X_l}(X_l) $. Corollary~\ref{cor_injSubsets} then implies $$ \inj_{X_{l+1}}(X_{l+1}) = \inj_{X_l}(X_{l+1}) = \inj_{X_l}(X_l) = \inj_X(X) $$ which completes the induction. We omit the index $ X_l $ in the following.
	
	Consider the sets $ \{C_t\}_{t\in I} $ between $ X_{l+1} $ and $ X_l $, obtained in Lemma~\ref{lem_soulSuperlevelSets}. Let $ g:I\to[0,\infty] $ as defined in (\ref{eq_g}).
	By assumption, we have $ \inj(X_l)=\inj(X)<D_K $ and thus, in case $ I=[a,0] $, set $ b:=0 $ and observe $ g(b)=\inj(X_l)<D_K $. In case $ I=[a,\infty) $, we use the fact that $ \bigcup_{t\in I}C_t = X_l $ and obtain $ \lim_{t\to\infty}g(t) = \inj(X_l)<D_K $ and thus there exists some arbitrarily large $ b\in I $ such that $ g(b)<D_K $. We show that $ g $ is constant on $ [a,b] $ and thus on $ I $. We obtain $ \inj(X_{l+1})=g(a) = g(b) =\inj(X_l) $, resp.\ $ \inj(X_{l+1})=g(a) = \lim_{t\to\infty}g(t) = \inj(X_l) $ which completes the proof.
	
	We define the level set $ J:=\{t\in [a,b]: g(t) = g(b)\} $.
	Since $ g(b)<D_K $, it agrees with the level set of the continuous function $ t\mapsto \min\{D_K,g(t)\} $ (Lemma~\ref{lem_gClosed}) and is thus closed.
	Assume $ J $ is not open, then there exists a sequence $ (t_i)_{i\in\IN} $ with limit $ t\in J $ and $ g(t_i)>g(b)=g(t) $. Since $ g $ is non\-increasing, we conclude $ t_i<t $. By Lemma~\ref{lem_gClosed}, there exists $ \delta>0 $ such that $ g $ is continuous on $ (t-\delta,t+\delta) $. Let $ n\in\IN $ be sufficiently large such that $ t-t_n<\delta $. Then the mean value theorem~\ref{prop_mvt} yields for the continuous function $ g|_{[t_n,t]} $ some point $ s\in (t_n,t] $ such that 
	$$ g^-(s)\leq \frac{g(t)-g(t_n)}{t-t_n} < 0. $$ But Lemma~\ref{lem_gLeftDerivativeZero} yields $ g^-(s)=0 $, a contradiction.
	In conclusion, $ J $ is a non\-empty, open and closed subset of the connected set $ [a,b] $ which implies that $ J = [a,b] $ and consequently $ g(t)=g(b) $ for all $ t\in [a,b]$.
\end{proof}

\begin{rem}\label{rem_alternativeProof}
In \cite{Sh76}, where Theorem~\ref{thm_injSoul} for Riemannian manifolds was initially stated, a proof was given using Sharafutdinov's retraction. A very similar approach also translates to Alexandrov spaces very elegantly and was kindly brought to our attention by Alexander Lytchak. For the sake of completeness we give a short outline of the idea.
As in the proof above, the Klingenberg Lemma~\ref{lem_Klingenberg} allows us to control the injectivity radius by controlling the length of closed local geodesics.
Suppose there exists a closed local geodesic $ \gamma $ which is not contained in the soul $ S $ but has length $ 2\inj(X) < 2\inj(S) $. We split it into two geodesics $ \gamma_1 $ and $ \gamma_2 $ of equal length and coinciding endpoints and let them ``flow'' into the soul under \v{S}arafutdinov's retraction $ \text{\v{S}}:=F(\cdot,1) $ (see Proposition~\ref{prop_sarafutdoinovsRetraction}) to obtain the curves $ \text{\v{S}} \circ \gamma_1  $ and $ \text{\v{S}} \circ \gamma_2 $ lying in the soul. By the non-expanding property of the deformation, their lengths cannot have increased. Applying a curve shortening flow to $ \text{\v{S}} \circ \gamma_1 $ and $ \text{\v{S}} \circ \gamma_2 $, we obtain two local geodesics $ \sigma_1 $ and $ \sigma_2 $ joining the same points. This can be done by replacing increasing parts of the curve with geodesics which vary continuously with their endpoints since they have length smaller than the injectivity radius (see Remark~\ref{rem_contGeodBelowInj}).
If $ \sigma_1 $ and $ \sigma_2 $ are distinct, their lengths give an upper bound for the injectivity radius and yield a contradiction.
Suppose otherwise that they coincide. Then $ \gamma $ is contractible and a null homotopy is given by \v{S}arafutdinov's deformation in concatenation with the curve shortening flow and the canonical null homotopy 
that contracts the obtained curve (see Figure~\ref{fig_alternativeProofSarafutdinov}). 
\begin{figure}[h]
	\centering
	\includegraphics[width=0.9\linewidth]{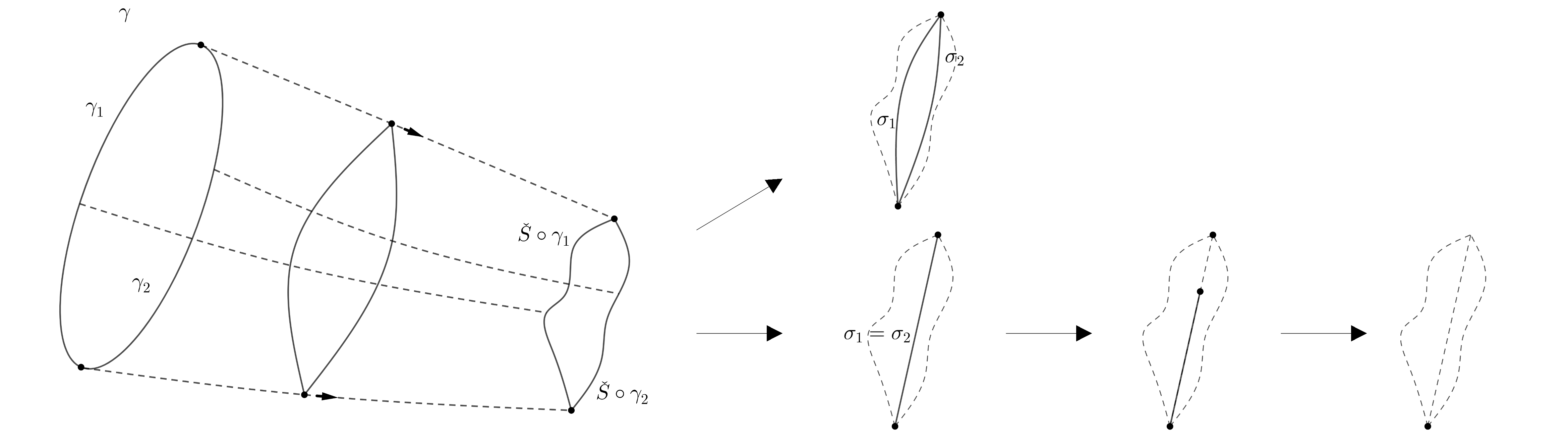}
	\caption{A sketch of the null homotopy of $ \gamma $}\label{fig_alternativeProofSarafutdinov}
\end{figure}
The lengths of all curves of this null homotopy have lengths not larger than that of $ \gamma $ which, by assumption, is smaller than $ 2 D_K $. We obtain the desired contradiction from \cite[Theorem 12.1]{ShSo} which states that such a null homotopy must contain a curve of length at least $ 2 D_K $.
\end{rem}
The proof given for Theorem~\ref{thm_injSoul} and the proof from Remark~\ref{rem_alternativeProof} provide different insight into the matter. The former is independent of \v{S}arafutdinov's retraction and its related gradient flow theory. It connects in an elementary manner to the soul construction and common tools of Alexandrov spaces. The latter exploits the non-expansiveness of the gradient flow in an elegant way. It is independent of continuity of the injectivity radius or the exact description of the superlevel sets of the soul construction.

\appendix
\section{Appendix}\label{sec_App}
\subsection{Hausdorff Distance}
\begin{defn}
	Let $ X $ be a metric space. We denote by $ \mathcal{K}(X) $ the set of all nonempty compact subsets of $ X $. Then the \textit{Hausdorff distance}\index{Hausdorff distance}\index{Metric!Hausdorff}
	\begin{equation*}
		d:\mathcal{K}(X)\times\mathcal{K}(X)\to [0,\infty),\quad (A,B)\mapsto \max\{\sup\limits_{a\in A} d(a,B), \sup\limits_{b\in B} d(A,b) \}
	\end{equation*}
	defines a metric on $ \mathcal{K}(X) $.
\end{defn}
\begin{lem}\label{lem_HausdorffImagemapCont}
	Let $ X,Y $ be metric spaces and $ f:X\to Y $ be continuous. If $ f $ is uniformly continuous or $ X $ is proper, then the map $$ \mathcal{K}(X)\to\mathcal{K}(Y),\quad A\mapsto f(A) $$ is uniformly continuous.
\end{lem}
\begin{proof}
	Let $ \varepsilon>0 $ and $ A\in\mathcal{K}(X) $. If $ f $ is not uniformly continuous, consider the set $ A':=\{p\in X: d(p,A)\leq 1\} $ which is bounded and closed and thus compact, otherwise set $ A':=X. $ The continuous function $ f $ is uniformly continuous when restricted to the compact set $ A' $. Hence, there exists a $ \delta>0 $ such that all $ a,b\in A' $ with $ d(a,b)<\delta $ satisfy $ d(f(a),f(b))<\eps $. If $ \delta>1 $, we choose $ \delta:=1 $. Fix $ B\in \mathcal{K}(X) $ with $ d_H(A,B)<\delta $. Note that $ B\subseteq A' $.
	Let $ a\in A $, then $ d(a,B)<\delta $ i.e.\ there exists $ b\in B $ such that $ d(a,b)<\delta $ and hence $$ d(f(a),f(B))\leq d(f(a),f(b))<\varepsilon. $$ This holds for all $ a\in A $. We obtain the analogous statement $ d(f(A),b)<\eps $ for all $ b\in B $ and obtain finally
	\begin{equation*}
		d_H(f(A),f(B)) =\max\left\{\sup\limits_{a\in A} d(f(a),f(B))\ ,\ \sup\limits_{b\in B} d(f(A),f(b)) \right\} <\varepsilon. \qedhere
	\end{equation*}
\end{proof}

\begin{lem}\label{lem_HausdorffMinContinuous}
	The map $ \mathcal{K}(\IR) \to \IR, \ A\mapsto \min A $ is continuous.
\end{lem}
\begin{proof}
	Let $ A,B\in \mathcal{K}(\IR) $ and without loss of generality $ \min A \leq \min B $. The distance of $ \min A $ to the set $ B $ is then realized by $ d(\min A , \min B) $ and consequently
	\begin{align*}
		d_H(A,B)&=\max\{\sup\limits_{a\in A} d(a,B), \sup\limits_{b\in B} d(A,b) \}\\
		&\geq \sup\limits_{a\in A} d(a,B) \geq d(\min A,B) = d(\min A , \min B),
	\end{align*} 
	which gives continuity of the map $ A\mapsto \min A $.
\end{proof}

\subsection{Real Functions with One-Sided Derivative}
\begin{defn}\label{def_onesdiedDerivative}
	Let $ g:[a,b]\to \IR $ be a function. The \textit{left derivative}\index{Derivative!left} in $ x\in(a,b] $ is defined to be the limit
	\begin{equation*}
		g^-(x):=\lim_{h\searrow 0}\frac{g(t)-g(t-h)}{h}
	\end{equation*}
	if it exists. For a function $ g:A\to\IR $ defined on a subset $ A\subseteq\IR $ we define the left derivative in $ x\in A $ if there is some $ \varepsilon>0 $ such that $ (x-\varepsilon,x]\subseteq A $ to be $ (g|_{[x-\varepsilon,x]})^-(x) $.
	
	The \textit{right derivative}\index{Derivative!right} in $ x\in[a,b) $ is defined to be the limit
	\begin{equation*}
		g^+(x):=\lim_{h\searrow 0}\frac{g(t+h)-g(t)}{h}
	\end{equation*}
	if it exists. For a function $ g:A\to\IR $ defined on a subset $ A\subseteq\IR $, we define the right derivative in $ x\in A $, if there is some $ \varepsilon>0 $ such that $ [x,x+\varepsilon]\subseteq A $, to be $ (g|_{[x,x+\varepsilon]})^-(x) $.
	
\end{defn}
The following is a generalization of the Mean Value Theorem for differentiable functions to functions with one-sided derivatives.
\begin{prop}[Mean Value Theorem]\label{prop_mvt}
	Let $ g:[a,b]\to \IR $ be a continuous function. If the left derivative $ g^- $ exists in $ (a,b] $, then there exist points $ c,d\in(a,b] $ such that
	$$ g^-(c) \leq \frac{g(b)-g(a)}{b-a} \leq g^-(d). $$
	If the right derivative $ g^+ $ exists in $ [a,b) $, then there exist points $ c,d\in[a,b) $ such that
	$$ g^+(c) \leq \frac{g(b)-g(a)}{b-a} \leq g^+(d). $$
\end{prop}
\begin{proof}
	We first prove a generalization of Rolle's Theorem where we consider a continuous function $ f :[a,b]\to \IR $ which admits left derivatives and satisfies $ f(a)=f(b)=0 $. The continuous function $ f $ assumes its maximum and minimum on the compact interval $ [a,b] $ at some points $ x_{min},x_{max}\in [a,b] $, respectively. We set $ c:=x_{max} $ and $ d:=x_{min} $. If $ x_{max}=a $, then both $ f(a)=f(b)=0 $ are maximal values of $ f $ and hence we reset $ x_{max}:=b $. Analogously, if $ x_{min}=a $, we reset $ x_{min}:=b $. From now on we can assume $ x_{min},x_{max}\in(a,b] $. Thus, we have $ f^-(x_{min})\leq 0 $ and $ f^-(x_{max})\geq 0 $ by definition of the one-sided difference quotient. We obtained $ c,d\in (a,b] $ that satisfy
	\begin{equation}\label{eq_mvt}
		f^-(c) \leq 0 \leq f^-(d).
	\end{equation}
	Now, for the general continuous function $ g $, set $$ f:[a,b]\to\IR,\quad t\mapsto g(t)-\frac{g(b)-g(a)}{b-a}(t-a) - g(a) $$ and observe $$ f^-(t)=g^-(t)-\frac{g(b)-g(a)}{b-a} $$ for $ t\in(a,b] $. Furthermore, $ f $ satisfies $ f(a)=f(b)=0 $ and applying (\ref{eq_mvt}) to $ f $ yields points $ c,d\in(a,b] $ such that
	\begin{equation*}
		g^-(c) =  f^-(c) + \frac{g(b)-g(a)}{b-a} \leq \frac{g(b)-g(a)}{b-a} \leq f^-(d) + \frac{g(b)-g(a)}{b-a} = g^-(d).
	\end{equation*}
	The proof for $ g^+ $ works analogously with the according changes of sign and interval bounds.
\end{proof}

\subsection{Klingenberg's Injectivity Radius Estimate}\label{sec_AppKlingenberg}
We give an alternate proof of \cite[Theorem~8.3]{ShSo} in view of incorrect arguments in the proof of \cite[Lemma~8.1]{ShSo} which was solely used to prove \cite[Theorem~8.3]{ShSo}. 

Let us first note the distinct definitions of conjugate points defined in Section~\ref{sec_Inj} and in \cite{ShSo}. We indicate the notions as defined in the latter by putting them into quotation marks. In \cite{ShSo} a stronger notion of convergence than induced by the uniform metric $ \rho $ is used. But under the assumptions of Proposition~\ref{prop_ShSo8.3}, they are equivalent by \cite[Lemma~1.11]{ShSo} and each ``one-sided conjugate point'' is a conjugate point. 
A counterexample to the converse is given by the following. It also provides an affirmative answer to \cite[Open~Question~4.5]{ShSo}.
\begin{ex}{(\cite{AB}).}
	Consider the closed hemisphere $ X\subseteq \IS^2 $. This is a closed convex subset of a Riemannian manifold of constant curvature $ 1 $ and thus itself a locally compact Alexandrov space of curvature $ \geq 1 $ and $ \leq 1 $. First, observe that all local geodesics of length larger than $ \pi $ are contained in the boundary circle $ \partial X $. Now, consider such a local geodesic $ \gamma:[0,1] \to X $. Any sequence of local geodesics converging uniformly to $ \gamma $ must eventually have lengths larger than $ \pi $ (Lemma~\ref{lem_factsAS}~\ref{lem_lengthContLocGeod}) and thus are eventually contained in $ \partial X $. We obtain the following consequences. Along $ \gamma $ the endpoint $ \gamma(1) $ is a conjugate point to $ \gamma(0) $ (in the sense of Definition~\ref{def_conj}) since a small open neighborhood of $ \gamma(0) $ is mapped onto a circle segment contained in the boundary circle, i.e.\ not an open neighborhood of $ \gamma(1) $. Furthermore, $ \gamma(0) $ and $ \gamma(1) $ are ``unreachable conjugate'' but not ``symmetric conjugate'' along $ \gamma $ (hence also not ``one-sided conjugate'').
\end{ex}
Lemma~\ref{lem_Klingenberg} is a version of \cite[Lemma~8.1]{ShSo} for the injectivity radius instead of the minimal radius and a different notion of conjugate points. \cite[Theorem~8.3]{ShSo} can be recovered using that the injectivity radius coincides with the minimal radius in this situation.
\begin{prop}[\text{\cite[Theorem~8.3]{ShSo}}]\label{prop_ShSo8.3}
	If $ X $ is a compact length space with a positive and finite minimal radius, then either there exists $ p $ and $ q $ which are ultimate conjugate and $ d(p,q)\leq \MinRad(X) $ or there is a closed local geodesic $ \gamma:\IS^1\to X $ which has length twice the minimal radius.
\end{prop}
\begin{proof}
	$ X $ is, in particular, complete and locally compact (and thus geodesic). Therefore, the results in Section~\ref{sec_Inj} apply. Furthermore, it is ``locally uniformly minimizing'' enabling \cite[Proposition~5.4]{ShSo} by which all conjugate points (in the sense of Definition~\ref{def_conj}) are ``ultimate conjugate points'' and hence $ ``\UltConj(p)"\leq \conj(p) $ for each $ p\in X $.\\
	Assume all pairs of points that are ``ultimate conjugate'' are more than $ \MinRad(X) $ apart. That is, $ \conj(X)\geq ``\UltConj(X)">\MinRad(X) $ and therefore $ \inj(X)=\MinRad(X) $ by Corollaries~\ref{cor_ineqInjMinRad} and \ref{cor_injMinRad}. Furthermore, the absence of ``ultimate conjugate points'' is sufficient for Corollary~\ref{cor_injCont} (with $ D_K $ replaced by $ ``\UltConj(X)" $) which yields that the map $ p\mapsto \min\{\inj(p),``\UltConj(X)"\} $ is continuous and assumes a minimum on the compact space $ X $ in some point $ p\in X $. Again, by compactness of $ X $, there exists a point $ q\in \overline{\Cut(p)} $ with $ d(p,q)=\inj(p) $ and Lemma~\ref{lem_Klingenberg} yields a local geodesic $ \tilde{\gamma} $ from $ p $ to $ p $ through $ q $. We obtain $ \inj(q)\leq\inj(p)=\inj(X)\leq \inj(q) $ and thus equality, i.e.\ $ p $ realizes the injectivity radius of $ q $. Furthermore, it holds $ d(q,p)=\inj(q)<\conj(X) $ which allows us to conclude analogously to the proof of Lemma~\ref{lem_Klingenberg} that $ \tilde{\gamma} $  is minimizing at $ p $, that is, one can define a map $ \gamma:\IS^1\to X $ that is a linear reparametrization of an isometry. It has length 
	\begin{equation*}
		L(\gamma)=2\inj(p)=2\inj(X)=2\MinRad(X). \qedhere
	\end{equation*}  
\end{proof}
\com{Is the previous statement true if one replaces $ \MinRad $ by $ \inj $? Yes but careful explanation is needed of why those two notions agree below ultimate conjugate points even without locally uniformly minimizing property, that is, without cite[Prop.~5.4]{ShSo}.}

\phantomsection\addcontentsline{toc}{section}{References} 
\bibliography{d:/Jona/Documents/HESSENBOX-DA/MyStuff/MyLit/000_bibs/injectivityRadiusSoul.bib}	
\end{document}